\documentclass[10pt]{article}

\usepackage{graphics}
\usepackage{graphicx}

\usepackage[a4paper, left=35mm,right=35mm,top=34mm,bottom=34mm]{geometry}
\usepackage[utf8]{inputenc}
\usepackage[T1]{fontenc}
\usepackage[english]{babel}

\usepackage{enumerate}
\usepackage{graphicx}
\usepackage{hyperref}
\hypersetup{
    colorlinks=true,
    linkcolor=blue,
    filecolor=magenta,      
    urlcolor=cyan,
}
\usepackage{lipsum}
\usepackage{amsfonts}
\usepackage{graphicx}
\usepackage{epstopdf}
\usepackage{algorithmic}

 \usepackage{dsfont}
 \usepackage{psfrag}
 \usepackage{color}
 \usepackage{tikz}
 \usepackage{subfigure}
\usepackage{mathtools,amsthm,amssymb,amsfonts}
\usepackage{algorithm}
\usepackage{algorithmic}
\makeatother
\theoremstyle{plain}
\def\del  {\partial}
\def\eps{\varepsilon}

\def\R{\mathbb{R}}

\def\N{\mathbb{N}}
\def\C{\mathbb{C}}

 \def\dx{{\rm d}x}
 \def\dy{{\rm d}y}

\newtheorem{remark}{\textbf{Remark}}
\newtheorem{theorem}{\textbf{Theorem}}
\newtheorem{lemma}{\textbf{Lemma}}
\newtheorem{definition}{\textbf{Definition}}

\usepackage{caption} 
\usepackage{fancyhdr}

\author{
  {\normalsize Fr\'ed\'eric Magoul\`es}\thanks{MICS, CentraleSup\'elec, Universit\'e Paris Saclay, 
3 rue Joliot Curie, 91190 Gif-sur-Yvette, France and 
Faculty of Engineering and Information Technology, University of Pécs, Pécs, Hungary.}
  \and
  {\normalsize Thi Phuong Kieu Nguyen}\thanks{MICS, CentraleSup\'elec, Universit\'e Paris-Saclay, France.}
	\and
	{\normalsize Pascal Omnes}\thanks{DES-Service de thermo-hydraulique et de m\'ecanique des fluides (STMF), CEA, Universit\'e Paris-Saclay, F-91191, Gif-sur-Yvette, France
and
Universit\'e Sorbonne Paris Nord, LAGA, CNRS, UMR 7539,  F-93430, Villetaneuse, France.}
	\and
  {\normalsize Anna Rozanova-Pierrat}\thanks{CentraleSup\'elec, Universit\'e Paris-Saclay, France
    (correspondence, anna.rozanova-pierrat@centralesupelec.fr).}
		}
\title{Optimal absorption of acoustic waves by a boundary}
\date{}

\begin{document}
\maketitle
\thispagestyle{fancy}

\begin{abstract}
\noindent In the aim to find the simplest and most efficient shape of a noise absorbing wall to dissipate the acoustical energy of a sound wave, we consider a frequency model described by the Helmholtz equation with a damping on the boundary. The well-posedness of the model is shown in a  class of domains with $d$-set boundaries ($N-1\le d<N$). We introduce a class of admissible Lipschitz boundaries, in which an optimal shape of the wall exists in the following sense: We prove the existence of a  Radon measure on this shape, greater than or equal to the usual Lebesgue  measure, for which the corresponding solution of the Helmholtz problem realizes the infimum of the acoustic energy defined with the Lebesgue measure on the boundary. If this Radon measure coincides with the Lebesgue measure, the corresponding solution realizes the minimum of the energy. 
For a fixed porous material, considered as an acoustic absorbent, we derive the damping parameters of its boundary  from the corresponding time-dependent problem described by the damped wave equation (damping in volume).
\end{abstract}


\section{Introduction}
The diffraction and absorption of waves by a system with both absorbing properties and irregular geometry is an open physical problem. 
This has to be solved to understand why anechoic chambers (electromagnetic or acoustic) do work better with irregular absorbing walls. The first studies relating irregular geometry and absorption are performed numerically in~\cite{FELIX-2007}.
Therefore there is a question about the existence of an optimal shape of an absorbent wall (for a fixed absorbing material), optimal in the sense that it is as dissipative as possible for a large range of frequencies, and at the same time that such a wall could effectively be constructed.
In the framework of the propagation of acoustic waves, the acoustic absorbent material of the wall is considered as a porous medium.
In this article, for a fixed frequency of the sound wave, we solve the shape optimization problem that consists in minimizing the acoustic energy for a frequency model with a damping on the boundary.
Then we extend this method in order to find an efficient shape for a finite range of frequencies.

In the area of optimization of acoustic performances of non absorbing walls, Duhamel~\cite{DUHAMEL-1998,DUHAMEL-2006} studies sound propagation in a two-dimensional vertical cut of a road wall and uses genetic algorithms to obtain optimal shapes (some of them are however not connected and thus could not be easily manufactured).
The author also uses a branch and bound (combinatorial optimization) type linear programming in order to optimize the sensors' positions that allow an active noise control, following former work introduced by Lueg~\cite{GUICKING-1990} in 1934.
Abe et al.~\cite{ABE-2010} consider a boundary elements based shape optimization of a non absorbing two-dimensional wall in the framework of a two-dimensional sound scattering problem for a fixed frequency (for the Helmholtz equation) using a topological derivative with the principle that a new shape or topology is obtained by nucleating small scattering bodies.
Also for the Helmholtz equation for a fixed frequency, using the shape derivative of a functional representing the acoustical energy, Cao and Stanescu~\cite{CAO-2002} consider a two-dimensional shape design problem for a non-absorbing part of the boundary to reduce the amount of noise radiated from aircraft turbofan engines.
For the same problem, Farhadinia~\cite{FARHADINIA-2010} developed a method based on measure theory, which does not require any information about gradients and the differentiability of the cost function.

On the other hand, for shape optimization problems there are  theoretical results, reviewed in Refs.~\cite{ALLAIRE-2007,MOHAMMADI-2010}, which rely on the topological derivatives of the cost functional to be minimized, with numerical application of the gradient method in both two and three dimensional cases (in the framework of solid mechanics).
In particular, Achdou and Pironneau~\cite{ACHDOU-1991} considered the problem of optimization of a photocell, using a complex-valued Helmholtz problem with periodic boundary conditions with the aim to maximize the solar energy in a dissipative region.
For acoustic waves in the two-dimensional case, optimization of the shape of an absorbing inclusion placed in a lossless acoustic medium was considered in Refs.~\cite{MUNCH-2009,MUNCH-2006}.
The considered model is the linear damped wave equation~\cite{COX-1994,ASCH-2003}. 
Using the topology derivative approach, M\"unch et al. consider in~\cite{MUNCH-2009,MUNCH-2006} the minimization of the acoustic energy of the solution of the damped wave equation at a given time $T>0$ without any geometric restrictions and without the purpose of the design of an absorbent wall. 
See also~\cite{ANTIL-2017} for the shape optimization of shell structure acoustics.
 
In this article, we study the two-dimensional shape optimization problem for a Helmholtz equation with a damping on the boundary, modeled by a complex-valued Robin boundary condition.
The shape of the damping boundary is to be found, in the aim to minimize the total acoustical energy of the system.
In Section~\ref{SecDP}, we introduce the frequency model and its time-dependent analogue with a dissipation on the boundary.
We analyze its dissipative properties and give the well-posedness results, due to~\cite{BARDOS-1994,GANDER-2007} for at least Lipschitz boundaries, but we generalize the results for the Helmholtz problem in a larger class of domains with $d$-set boundaries using~\cite{ARFI-2017,ROZANOVA-PIERRAT-2020} (see~Appendix~\ref{AnnexB}). This class, named in~\cite{ARFI-2017} "admissible domains" and containing for instance the Von Koch fractals,  is composed  of all Sobolev extension domains, thanks to results of~\cite{HAJLASZ-2008}, with boundaries on which it is possible to define a surjective linear continuous trace operator with linear continuous right inverse.
However, for the shape optimization problem, only the Lipschitz boundary case is  considered here. 

We compare the frequency model with dissipation by the boundary to the corresponding model with a dissipation in the volume.
Dissipation in the volume is described by a damped wave equation in which the values of the coefficients for a given porous medium are given as functions of its macroscopic parameters (as porosity, tortuosity and resistivity to the passage of air), as proposed by~\cite{HAMET-1993}.
In particular, in~Theorem~\ref{ThPascal} we propose a possible way to find the complex parameter in the Robin boundary condition of the former model that best approximates the latter.
In Section~\ref{SecShapeDPr}, for the case of uniform Lipschitz boundaries satisfying a uniform $\eps$-cone property for a fixed $\eps>0$ and in addition having a uniform upper bound for their lengths inside non trivial balls (and hence which have boundaries with uniformly bounded lengths), we introduce the class of (shape) admissible domains, adding, as in the classical framework of shape optimization, the assumptions that all moving parts of the boundary belong to a compact set and that all domains having a fixed volume are included in a fixed bounded open set. In this class of admissible domains, for any fixed frequency we obtain the existence of an optimal shape in the sense that there exists a positive measure $\mu^*$ on the optimal shape $\Gamma$, equivalent to the usual Lebesgue measure $\lambda$, such that $\mu^*\ge \lambda$ and such that the weak solution of the corresponding Helmholtz problem realizes the infimum of the acoustical energy, the latter being defined using $\lambda$ on the boundary. In the case where $\mu^*=\lambda$ (this depends on the properties of the minimizing sequence which we don't know in advance), then it  is also  the minimum.  Moreover, we notice that in order to have the existence of an optimal shape %
 in a higher dimensional case (for instance in $\R^3$ or simply in $\R^N$)  it is sufficient to replace everywhere the $N-1$-dimensional Lebesgue measure of the boundary by the $N-1$-dimensional Hausdorff measure, since in that case the Lebesgue measure of the $N-1$-dimensional boundary is not equal to the Hausdorff measure as for one dimensional curves, but proportional to it (see~\cite[Thrm.~1.12, p.~13]{FALCONER-1985}, for the optimization in $\R^3$ the Lebesgue measure of the boundary is equal to $\pi/4$ times the Hausdorff measure). See also Ref.~\cite{BUCUR-2016} for a free discontinuity approach to a class of shape optimization problems involving a Robin condition on a free boundary.

 To summarize, the rest of the paper is organized as follows.  Section~\ref{SecDP} introduces a damped acoustical propagation model in which damping occurs through the boundary. It is described by the Helmholtz problem with a Robin boundary condition with a complex coefficient $\alpha$, for which we give a well-posedness result on an admissible domain in $\R^N$ in the sense of Ref.~\cite{ARFI-2017}. The existence of an optimal shape in the introduced acoustical framework is proved in~Section~\ref{SecShapeDPr}.  We recall useful results from~\cite{ARFI-2017} in~Appendix~\ref{AnnexB}. Appendix~\ref{Annex} details how to obtain the damping parameter $\alpha$ in the Robin boundary condition that best approximates a given model with dissipation in the volume.

\section{The model: motivation and known properties}\label{SecDP}

To describe the acoustic wave absorption by a porous medium, there are two possibilities.
The first one is to consider wave propagation in two media, typically air and a wall, which corresponds to a damping in the volume.
The most common mathematical model for this is the damped wave equation~\cite{ASCH-2003}.
The second one is to consider only one lossless medium, air, and to model energy dissipation by a damping condition on the boundary. 
In both cases, we need to ensure the same order of energy damping corresponding to the physical characteristics of the chosen porous medium as its porosity $\phi$, tortuosity $\alpha_h$ and resistivity to the passage of air $\sigma$~\cite{HAMET-1993}.

Thanks to Ref.~\cite{HAMET-1993}, we can define the coefficients in the damped wave equation (damping in volume) as functions of the above mentioned characteristics.
More precisely, for  a regular bounded domain $\Omega\subset \R^{2}$ (for instance $\del \Omega\in C^1$) composed of two disjoint parts $\Omega=\Omega_0\cup \Omega_1$ of two homogeneous media, air in $\Omega_0$ and a porous material in $\Omega_1$, separated by an internal boundary $\Gamma$, we consider the following boundary value problem (for the pressure of the wave) 
\begin{equation}\label{amortih}
 \displaystyle\left\{ \begin{array}{l}
 \xi(x)\del_t^2 u+a(x)\del_t u-\nabla \cdot (\eta(x) \nabla u)=0\quad\hbox{in  } \Omega,\\
\frac{\del u}{\del n}|_{\R_t\times\del \Omega}\equiv 0,\quad [u]_{\Gamma}=[\eta\nabla u \cdot n]_{\Gamma}=0,
\\
u|_{t=0}=u_0\mathds{1}_{\Omega_0}, \quad \del_t u|_{t=0}=u_1\mathds{1}_{\Omega_0},%
\end{array}\right.
\end{equation}
with $\xi(x)=\frac{1}{c_0^2}$, $a(x)=0$, $\eta(x)=1$ in air, $i.e.$, in $\Omega_0$, and
$$\xi(x)=\frac{\phi\gamma_p}{c_0^2}, \quad a(x)=\sigma \frac{\phi^2 \gamma_p}{c_0^2 \rho_0 \alpha_h}, \quad\eta(x)=\frac{\phi}{\alpha_h}$$
in the porous medium, $i.e.$, in $\Omega_1$.
The external boundary $\del \Omega$ is supposed to be rigid, $i.e.$, Neumann boundary condition are applied, and on the internal boundary $\Gamma$ we have no-jump conditions on $u$ and $\eta\nabla u \cdot n$, where $n$ denotes the normal unit vector to~$\Gamma$.
Here,  $c_0$ and $\rho_0$ denote respectively the sound velocity and the density of air, whereas $\gamma_p=7/5$ denotes the ratio of specific heats.
But instead of energy absorption in volume, we can also consider the following frequency model of damping by the boundary.
Let $\Omega$ be a connected bounded domain of~$\R^{2}$ with a Lipschitz boundary $\del \Omega$.
We suppose that  the boundary $\del \Omega$ is divided into three parts $\del \Omega=\Gamma_D\cup\Gamma_N\cup \Gamma$ (see Fig.~\ref{FigGD} for an example of $\Omega$, chosen for the numerical calculations) and consider
\begin{equation}\label{Helmholtz}
\left\{ \begin{array}{l} \triangle u+\omega^2  u=f(x), \quad x\in \Omega,\\
u=g(x)\quad \hbox{on }\Gamma_D,\quad
\dfrac{\del u}{\del n} =0\quad \hbox{on }\Gamma_N,\quad \dfrac{\del u}{\del n} +\alpha(x) u=\mathrm{Tr}h(x)\quad \hbox{on }\Gamma, \end{array} \right.
	\end{equation}
where $\alpha(x)$ is a complex-valued regular function with a strictly positive real part ($\mathrm{Re}(\alpha)>0$) and a strictly negative imaginary part ($\mathrm{Im}(\alpha)<0$).
\begin{remark}
This particular choice of the signs of the real and the imaginary parts of $\alpha$ are needed for the well-posedness properties \cite{GANDER-2007} and the energy decay of the corresponding time-dependent problem.
In addition, as the frequency $\omega>0$ is supposed to be fixed, $\alpha$ can contain a dependence on $\omega$, $i.e.$, $\alpha\equiv \alpha(x,\omega).$
\end{remark}

Problem~(\ref{Helmholtz}) is a frequency version of the following time-dependent wave propagation problem with $U(t,x)=e^{-i\omega t} u(x)$, considered in Ref.~\cite{BARDOS-1994} for $g=0$ on $\Gamma_D$:
\begin{align}
  &  \del_t^2U-\triangle U=-e^{-i\omega t} f(x),\label{ondeba}\\
  & U|_{t=0}=U_0, \quad \del_t U|_{t=0}=U_1,\label{inic}\\
 &U|_{\Gamma_D}=g, \quad \left.\frac{\del U}{\del n} \right|_{\Gamma_N}=0,\\
&\frac{\del U}{\del n}-\frac{{\mathrm{Im}}(\alpha(x))}{\omega} \del_t U+{\mathrm{Re}}(\alpha(x)) U|_{\Gamma}=0.\label{Robintemp}
\end{align}
To show the energy decay, we follow \cite{BARDOS-1994} and introduce the Hilbert space $X_0(\Omega)$, defined as the Cartesian product of the set of functions $u\in H^1(\Omega)$, which vanish on~$\Gamma_D$ with the space $L_2(\Omega)$.
The equivalent norm on $X_0(\Omega)$ is defined by
$$
 \|(u,v)\|^2_{X_0(\Omega, \mu )}=\int_\Omega \left(|\nabla_x u|^2+|v|^2\right)\dx+ \int_\Gamma {\mathrm{Re}}(\alpha(x))  |u|^2 {\mathrm{d}}  \mu 
$$
with the corresponding inner product
\begin{equation}\label{innerPrH}
 \langle(u_1,u_2),(v_1,v_2)\rangle=\int_\Omega \left(\nabla_x u_1\nabla_x v_1+ u_2 v_2\right)\dx+ \int_\Gamma {\mathrm{Re}}(\alpha(x))  u_1 v_1 {\mathrm{d}}  \mu .
\end{equation}
 Here $\mu$ is a  Radon positive measure on $\Gamma$, which in the case of a regular $\Gamma$ (at least Lipschitz), if there are no specific assumptions, is equal to the Lebesgue measure on $\Gamma$, and in this case is denoted by $\lambda$.

The advantage of this norm is that the energy balance of the homogeneous problem~(\ref{ondeba})--(\ref{Robintemp}) has the form
$$
\del_t \left(\|(U,\del_t U)\|^2_{X_0(\Omega, \mu )}\right)=\frac{2}{\omega}\int_\Gamma {\mathrm{Im}}(\alpha(x)) |\del_t U|^2 {\mathrm{d}}  \mu .
$$
Therefore, for $\mathrm{Im}(\alpha)<0$ on $\Gamma$, the energy decays in time.
For the case of a smooth boundary $\del \Omega$ (at least Lipschitz), we have the well-posedness of both models.
Thanks to \cite{BARDOS-1994}, for all $f\in L_2(\Omega)$, $(U_0,U_1)\in X_0(\Omega)$ there exists a unique solution $(U,U_t)\in C(]0,\infty[,X_0(\Omega))$ of system~(\ref{ondeba})--(\ref{Robintemp}) under the assumption that $\mathrm{Re}(\alpha(x))>0$ and $\mathrm{Im}(\alpha(x))<0$ are continuous functions. 

For the frequency model~(\ref{Helmholtz}) it is possible to generalize the weak well-posedness result in domains with Lipschitz boundaries~\cite{GANDER-2007} to domains with a more general  class of boundaries, named Ahlfors $d$-regular sets or simply $d$-sets~\cite{JONSSON-1984} (see~Appendix~\ref{AnnexB}), using functional analysis tools on ``admissible domains'' developed in~\cite{ARFI-2017}. The interest of this generalization is that this class of domains is optimal in the sense that it is the largest possible class~\cite{ARFI-2017} which keeps the Sobolev extension operators, for instance $H^1(\Omega)$ to $H^1(\R^N)$, continuous. In what follows, for our well-posedness result we take $\Omega\subset \R^N$, $N\ge 2$. 

We use, as in Refs.~\cite{ARFI-2017,ROZANOVA-PIERRAT-2020}, the existence of a $d$-dimensional ($0<d\le N$, $d\in \R$) measure $\mu$ equivalent or equal to the Hausdorff measure $m_d$ on $\del \Omega$ (see~Definition~\ref{Defdset})
and a generalization of the usual trace theorem~\cite{JONSSON-1984} (see~Appendix~\ref{AnnexB}) and the Green formula~\cite{LANCIA-2002,ARFI-2017} in the sense of the Besov space $B_\beta^{2,2}(\del \Omega)$ with $\beta=1-\frac{N-d}{2}>0$ constructed on $\del \Omega$ with the help of the measure $\mu$ (for the definition of the Besov spaces on $d$-sets see Ref.~\cite{JONSSON-1984} p.135 and Ref.~\cite{WALLIN-1991}).
Note that for $d=N-1$, one has $\beta=\frac{1}{2}$ and
$
B_\frac{1}{2}^{2,2}(\del \Omega)=H^\frac{1}{2}(\del \Omega)
$
as usual in the case of a Lipschitz boundary. In what follows we write $L_2(\del \Omega, \mu)$ to specify  that the space is defined with respect to the measure $\mu$. Some main elements of functional analysis on $d$-sets are presented in~Appendix~\ref{AnnexB}.

Moreover, we stress that once a measure $\mu$ is fixed on 
the boundary $\del \Omega$, it modifies the meaning of the Green formula in the following sense: 
for all $u$ and $v$ from $H^1(\Omega)$ with $\Delta u\in L_2(\Omega)$ the normal derivative of $u$ is understood as the linear continuous functional on the Besov space $B^{2,2}_{\beta}(\del \Omega)$ constructed by $\mu$ according to the definition %
\begin{equation*}
 \langle \frac{\del u}{\del \nu}, 
\mathrm{Tr}v\rangle _{((B^{2,2}_{\beta}(\del \Omega))', B^{2,2}_{\beta}(\del \Omega))}:=\int_\Omega v\Delta u\dx + \int_\Omega \nabla v \cdot \nabla u \dx.
\end{equation*}

Considering the Helmholtz  problem~(\ref{Helmholtz}), we introduce the Hilbert space 
 \begin{equation}\label{EqVOM}
  V(\Omega)=\{u\in H^1(\Omega)|\; u=0 \hbox{ on } \Gamma_D\}
 \end{equation}
equipped with the norm (equivalent  to the usual norm $\|\cdot\|_{H^1(\Omega)}$) thanks to~Theorem~\ref{ThGENERIC}
$$
\|u\|^2_{V(\Omega,\mu)}=\int_\Omega |\nabla u|^2\dx+\int_\Gamma \mathrm{Re}(\alpha) |u|^2d \mu,
$$
and obtain the following well-posedness result: %
\begin{theorem} \label{ThWPc} %
Let $\Omega\subset \R^N$ be a bounded admissible domain with a compact $d$-set boundary ($N-2< d<N$) with a $d$-measure $\mu$ in the sense of~Theorem~\ref{ThGENERIC},
$\del \Omega=\Gamma_D\cup\Gamma_N\cup \Gamma$ and $\beta=1-(N-d)/2>0$. Let $\Gamma_D$ be also a $d$-set with the same properties as $\del \Omega$ itself.
Let in addition $\mathrm{Re}(\alpha(x))>0$, $\mathrm{Im}(\alpha(x))<0$ be  smooth  functions (at least continuous) on~$\Gamma$. 
Then for all $f\in L_2(\Omega)$, $g\in B^{2,2}_{\beta}(\Gamma_D)$, $h\in H^1(\Omega)$, and $\omega>0$ there exists 
a unique solution $u$ of the Helmholtz problem~(\ref{Helmholtz}), such that $(u-\tilde{g}) \in V(\Omega)$ (where $\tilde{g}$ is a lifting in $H^1(\Omega)$ of the boundary data $g$)
in the following sense: for all $v\in V(\Omega)$
\begin{equation}\label{eqVFHfh}
  \int_{\Omega}\nabla u \cdot \nabla \bar{v} \dx  - \omega^2\int_{\Omega}u\bar{v} \dx +\int_{\Gamma}\alpha \, \mathrm{Tr} \, u \, \mathrm{Tr}\,\bar{v} \,d \mu=-\int_{\Omega}f\bar{v} \dx  +\int_{\Gamma}\mathrm{Tr}\,h \, \mathrm{Tr}\,\bar{v} \, d \mu.
\end{equation}
Moreover, the solution of problem~(\ref{Helmholtz}) $u\in  H^1(\Omega)$, continuously depends on the data:
there exists a constant $C>0$, depending only on $\alpha$, $\omega$ and on $C_P(\Omega)$, such that
\begin{equation}\label{EqApriori}
 \|u\|_{H^1(\Omega)}\le C\left(\|f\|_{L_2(\Omega)}+\|g\|_{B^{2,2}_{\beta}(\Gamma_D)}+\|h\|_{H^1(\Omega)}\right),
\end{equation}
where $C_P(\Omega)$ is the Poincar\'e constant associated to $\Omega$.
 In particular, taking $g=0$, for all fixed $\omega>0$ the operator $$B: L_2(\Omega)\times H^1(\Omega) \to V(\Omega), \hbox{ defined by } B(f,h)=u$$ with $u$, the weak solution of~\eqref{eqVFHfh}, is a linear compact operator.

In addition, if, for $m\in \N^*$, $\del \Omega\in C^{m+2}$ (hence $\mu$ is an $N-1$-measure on $\del \Omega$), $f\in H^m(\Omega)$ and $g\in H^{m+\frac{3}{2}}(\Gamma_D)$, then the solution $u$ belongs to $H^{m+2}(\Omega)$. 
  
\end{theorem}
\begin{proof}

Let us start with $g=0$ and in addition suppose that $\alpha$ is a constant (the   generalization for $\alpha(x)$ is straightforward). 
By the linearity of the Helmholtz problem~\eqref{Helmholtz}, we set 
$u=u^f+u^h$, where $u^f$ is the solution of the Helmholtz problem with $\operatorname{Tr}h=0$ on $\Gamma$ and $u^h$ is the solution of the Helmholtz problem with $f=0$.

 When $h=0$, the variational formulation for $u^f$ becomes: for all $\phi\in V(\Omega)$
$$( u^f,\phi)_{V(\Omega,\mu)}-\omega^2 ( u^f,\phi)_{L_2(\Omega)}+i\operatorname{Im} \alpha ( \operatorname{Tr} u^f,\operatorname{Tr}\phi)_{L_2(\Gamma,\mu)}=-( f,\phi)_{L_2(\Omega)},$$
where  we have defined the following  equivalent inner product on $V(\Omega)$:
$$\forall (v,w)\in V(\Omega)\times V(\Omega)\quad (v,w)_{V(\Omega,\mu)}=(\nabla v, \nabla w)_{L_2(\Omega)}+\operatorname{Re} \alpha(\operatorname{Tr}v,\operatorname{Tr}w)_{L_2(\Gamma,\mu)}.$$

Hence,  the Riesz representation Theorem ensures the existence of a linear bounded operator $A: L_2(\Omega) \to V(\Omega)$ such that for $v\in L_2(\Omega)$
\begin{equation}\label{EqOpA}
 \forall \phi\in V(\Omega)\quad ( v,\phi)_{L_2(\Omega)}=( Av,\phi)_{V(\Omega,\mu)}
\end{equation}
and in addition, by the Poincar\'e inequality,
\begin{align*}
 \|Av\|_{V(\Omega,\mu)}&=\sup_{\|\phi\|_{V(\Omega,\mu)}=1} |(v,\phi)_{L_2(\Omega)}|
\le \sup_{\|\phi\|_{V(\Omega,\mu)}=1} \|v\|_{L_2(\Omega)}\|\phi\|_{L_2(\Omega)}\\
&\le C_P(\Omega) \sup_{\|\phi\|_{V(\Omega,\mu)}=1} \|v\|_{L_2(\Omega)}\|\phi\|_{V(\Omega,\mu)} = C_P(\Omega) \|v\|_{L_2(\Omega)}
\end{align*}
ensuring that $\|A\|\le C_P(\Omega)$.

In the same way, using the Riesz representation Theorem we also define a linear bounded operator $\hat{A}: L_2(\Gamma,\mu) \to V(\Omega)$ such that for $w \in L_2(\Gamma,\mu)$
$$\forall \phi\in V(\Omega)\quad (w ,\operatorname{Tr}\phi)_{L_2(\Gamma,\mu)}=( \hat{A} w,\phi)_{V(\Omega,\mu)}.$$
Indeed, it is sufficient to notice that for a fixed  $w \in L_2(\Gamma,\mu)$ the form $\ell:\phi\in V(\Omega) \mapsto \ell(\phi)=(w,\operatorname{Tr}\phi)_{L_2(\Gamma,\mu)}\in \C$  is linear and continuous  on $V(\Omega)$:
 \begin{equation*}
  \left|(w,\operatorname{Tr}\phi)_{L_2(\Gamma,\mu)} \right|\le \|w\|_{L_2(\Gamma,\mu)} \|\operatorname{Tr} \phi\|_{L_2(\Gamma,\mu)}\le  C\|\phi\|_{V(\Omega,\mu)}, 
 \end{equation*}
thanks to the continuity and the linearity of the trace from $V(\Omega)$ to $L_2(\Gamma,\mu)$.
Moreover, $\|\hat{A}\|\le C(\mathrm{Re}(\alpha))$ since
\begin{align*} 
 \|\hat{A} w \|_{V(\Omega,\mu)}&=\sup_{\|\phi\|_{V(\Omega,\mu)}=1} |(w,\operatorname{Tr}\phi)_{L_2(\Gamma,\mu)}|
\le \sup_{\|\phi\|_{V(\Omega,\mu)}=1} \| w \|_{L_2(\Gamma,\mu)}\|\operatorname{Tr}\phi\|_{L_2(\Gamma,\mu)}\\
&\le C(\mathrm{Re}(\alpha))\sup_{\|\phi\|_{V(\Omega,\mu)}=1} \| w \|_{L_2(\Gamma,\mu)}\|\phi\|_{V(\Omega,\mu)}=C(\mathrm{Re}(\alpha)) \| w \|_{L_2(\Gamma,\mu)}.
\end{align*}

Thus, denoting by $S$ the compact embedding operator of $V(\Omega)$ in $L_2(\Omega)$ (by the Poincar\'e inequality it holds that $\|S\|\le C_P(\Omega)$), the variational formulation can be rewritten in the following form:
\begin{equation}\label{eq:fv_operateurs}
\forall \phi\in V(\Omega)\quad \left((Id-\omega^2 A\circ S  + i \operatorname{Im} \alpha \hat{A}  \circ \operatorname{Tr}) u^f,\phi\right)_{V(\Omega,\mu)}=(-Af,\phi)_{V(\Omega,\mu)}.
\end{equation}
Thanks to the compactness of the trace operator $\operatorname{Tr}: V(\Omega)\to L_2(\del \Omega)$~\cite{ARFI-2017} (with $\|\operatorname{Tr}\|\le C(\mathrm{Re}(\alpha))$), the operator $T=A\circ S-i\frac{\operatorname{Im} \alpha}{\omega^2}\hat{A}  \circ \operatorname{Tr}: V(\Omega) \to V(\Omega)$ is  compact as a composition of continuous and compact operators (with $\|T\|\le C(\omega,\alpha,C_P(\Omega))$).
Thanks to the Fredholm alternative, it is then sufficient to prove that for $(h,f)=(0,0)$, then the unique solution is $u=0$, and this will allow us to conclude to the well-posedness of \eqref{eq:fv_operateurs}.
Setting $f=0$ in \eqref{eq:fv_operateurs}, choosing $\phi = u^f$ and separating real and imaginary parts of the equality, we first obtain that $\operatorname{Tr} u^f = 0$ on $\Gamma$ (since $|\operatorname{Im} \alpha| > 0$). By the Robin boundary condition on $\Gamma$, we then obtain that $\frac{\partial u}{\partial n} = 0$ on~$\Gamma$ (in the sense of a continuous linear functional on $B^{2,2}_{\beta}(\Gamma)$). 
Then, $u^f = 0$ in $\Omega$ follows by the uniqueness of the solution to the Cauchy problem for $\Delta + \omega^2 Id$ in the connected domain~$\Omega$ with Cauchy data on $\Gamma$ (see for example \cite[Theorems 1.1 and 1.2]{DARDE-2010}, which can be directly adapted to the case of a domain~$\Omega$ with a $d$-set boundary satisfying the conditions of Theorem~\ref{ThWPc} thanks to Theorem~\ref{ThGENERIC}). The operator $(Id-\omega^2 T)^{-1}$ is thus well defined and is also a linear continuous operator, by the Fredholm alternative theorem. 
Thus, we obtain
$$\|u^f\|_{V(\Omega,\mu)}\le  \frac{\|A\|}{\| Id -\omega^2T\|}\|f\|_{L_2(\Omega)} \le C(\omega,\alpha,C_P(\Omega))\|f\|_{L_2(\Omega)}.$$

In the same way, when $f=0$, the solution $u^h$ in $V(\Omega)$ satisfies the following variational formulation:
$$\forall \phi\in V(\Omega)\quad \left((Id-\omega^2 A\circ S  + i \operatorname{Im} \alpha \hat{A}  \circ \operatorname{Tr}) u^f,\phi\right)_{V(\Omega,\mu)}=(\hat{A}\circ \operatorname{Tr}h,\phi)_{V(\Omega,\mu)}.$$
Hence, as previously, we have
$$\|u^h\|_{V(\Omega,\mu)}\le \frac{\|\hat{A}\|\|\operatorname{Tr}\| }{\| Id -\omega^2T\|}\|h\|_{H^1(\Omega,\mu)} \le C(\omega,\alpha,C_P(\Omega)) \|h\|_{H^1(\Omega,\mu)}.$$

Consequently, we have proved the well-posedness and estimate~\eqref{EqApriori} for $g=0$. Hence, by the standard lifting  procedure, we can for instance chose $\tilde{g}\in H^1(\Omega)$ as the unique weak solution of~\eqref{Helmholtz} with $\omega=f=0$, with the homogeneous Neumann boundary conditions on $\Gamma_N$ and $\Gamma$ and with $\tilde{g}|_{\Gamma_D}=g$, to obtain the result with $g\ne 0$. 

To prove that the regularity of the boundary improves the regularity of the solution, we follow the classical approach, explained for elliptic equations in~\cite{EVANS-2010} Theorem~5 p.~323.

 The linearity and the continuity of $B$ are evident and equivalent to estimate \eqref{EqApriori}. Let us prove that for any fixed $\omega>0$, $B$ is also compact (see also Ref.~\cite{ARFI-2017} for the real Robin boundary condition).
 Indeed, let 
 $(f_j, h_j)\rightharpoonup (f,h)$ in $L_2(\Omega)\times H^1(\Omega)$. Taking for all $j\in \N$, $u_j=B(f_j, h_j)$ and $u=B(f, h)$, by the linearity and the continuity of $B$ it follows that $u_j\stackrel{V(\Omega)}{\rightharpoonup} u$.
 Knowing in addition that $\operatorname{Tr} : V(\Omega) \to L_2(\Gamma,\mu)$ and the inclusion of $H^1(\Omega)$ in $L_2(\Omega)$ are compact (see Ref.~\cite{ARFI-2017}) we have that $\operatorname{Tr} u_j \to \operatorname{Tr} u$ in $L_2(\Gamma,\mu)$ and $u_j \to u$ in $L_2 (\Omega)$.
 Choosing $v=u_j$ in the variational formulation \eqref{eqVFHfh} we find
 \begin{equation}
  \label{VarFormUj}
  \|u_j\|^2_{V(\Omega,\mu)}=\omega^2\|u_j\|^2_{L_2(\Omega)}-i\int_\Gamma \operatorname{Im} \alpha|\operatorname{Tr} u_j|^2d\mu-\int_\Omega f_j\overline{u}_j\dx+\int_\Gamma \operatorname{Tr} h_j \overline{\operatorname{Tr} u_j}d\mu,
 \end{equation}
and  hence,
 \begin{align*}
  \lim_{j\to +\infty}\|u_j\|^2_{V(\Omega,\mu)}=&\omega^2\|u\|^2_{L_2(\Omega)}-i\int_\Gamma \operatorname{Im} \alpha|\operatorname{Tr} u|^2d\mu-\int_\Omega f\overline{u}\dx+\int_\Gamma \operatorname{Tr} h \overline{\operatorname{Tr} u}d\mu \\=\|u\|^2_{V(\Omega,\mu)}.
 \end{align*}
 Having both $u_j\rightharpoonup u$ in $V(\Omega)$ and $\|u_j\|_{V(\Omega,\mu)}\to \|u\|_{V(\Omega,\mu)}$ implies that $u_j \to u$ in $V(\Omega)$ and hence $B$ is compact.
 Since the norm $\|u\|^2_{V(\Omega,\mu)}$ on $V(\Omega)$ is equivalent~\cite{ARFI-2017} to the norm
 $$\|u\|_{J}^2=A\|u\|^2_{L_2(\Omega)}+B\|\nabla u\|^2_{L_2(\Omega)}+C\|u\|^2_{L_2(\Gamma,\mu)},$$ the operator $B$ is also compact with respect to this norm. 
\end{proof}

In order to relate the model with a damping on the boundary and the model with a damping in the volume, we propose in~Appendix~\ref{Annex} a new theorem to identify the parameter $\alpha$ in the Robin boundary condition (Theorem~\ref{ThPascal}).
This parameter provides the best approximation (in some error minimizing sense) of the latter model by the former, in the case of a flat boundary $\Gamma$.

\section{Shape design problem}\label{SecShapeDPr}

We consider the two dimensional shape design problem, which consists in optimizing the shape of $\Gamma$ with the Robin dissipative condition in order to minimize the acoustic energy of system~\eqref{Helmholtz}.
The boundaries with the Neumann and Dirichlet conditions $\Gamma_D$ and $\Gamma_N$ are supposed to be fixed. 

We also define  a fixed open set~$D$ with a Lipschitz boundary which contains all domains $\Omega$.
Actually, as only a part of the boundary (precisely $\Gamma$) changes its shape, we  also impose that
the changing part always lies inside of the closure of a fixed open set~$G$ with a Lipschitz boundary: $\Gamma\subset \overline{G}$. %
\begin{figure}[!ht]
\begin{center}
\psfrag{Om}{$\Omega$}
\psfrag{G}{$G$}
  \psfrag{D}{$D$}%
  \psfrag{Gn}{$\Gamma_N$}%
  \psfrag{Gd}{$\Gamma_D$}
  \psfrag{Ga}{$\Gamma$}%
     \includegraphics[scale=0.9]{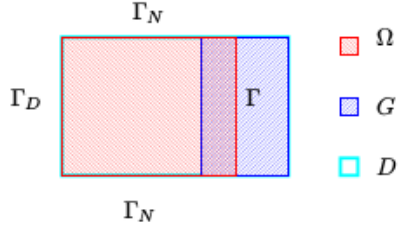}
\vspace*{8pt}
\caption{\label{FigGD} Example of a domain $\Omega$ in $\R^2$ with three types of boundaries: $\Gamma_D$ and $\Gamma_N$ are fixed and $\Gamma$ can be changed in the restricted area $\overline{G}$. Here $\Omega\cup G=D$ and obviously $\Omega\subset D$.}
\end{center}
\end{figure}
The set $G$ forbids $\Gamma$ to be too close to $\Gamma_D$, making the idea of an acoustical wall more realistic.

To introduce the class of admissible domains, on which we minimize the acoustical energy of system~(\ref{Helmholtz}), we define
$\mathcal{L}ip$  as the class of all domains $\Omega\subset D$ for which
\begin{enumerate}
	\item there exists a fixed $\eps>0$ such that  all
domains $\Omega\in \mathcal{L}ip$  satisfy the $\eps$-cone property~\cite{AGMON-1965,CHENAIS-1975}:
for all $x\in \partial\Omega$, there exists $\xi_x\in \R^{2}$  with  $\|\xi_x\|=1$  such that for all $y \in \overline{\Omega}\cap B(x,\eps)$
$$  C(y,\xi_x,\eps)=\{z\in \R^{2}| (z-y,\xi_x)\ge \cos(\eps)\|z-y\| \hbox{ and } 0<\|z-y\|<\eps\}\subset \Omega.$$
\item  there exists a fixed $\hat{c}>0$ such that for any $\Omega \in \mathcal{L}ip$ and  for all $x\in \Gamma$
we have 
\begin{equation}\label{EqSasha}
	\int_{\Gamma \cap B(x,r)}d\lambda \le \hat{c}r,
\end{equation}
where $B(x,r)$ is the open Euclidean ball centered in $x$ with radius $r$ and $\lambda$ is the usual one-dimensional Lebesgue measure on $\Gamma$.
\end{enumerate}
The uniform $\eps$-cone property implies, by Remark~2.4.8~\cite[p.~55]{HENROT-2005} and Theorem~2.4.7, that all boundaries of $\Omega\in \mathcal{L}ip$ are uniformly Lipschitz.

Let us notice that, by the boundedness of $D$ containing all $\Omega$, condition~(\ref{EqSasha}) implies that all $\Gamma$ for $\Omega\in \mathcal{L}ip$ have uniform length: there exists $M>0$ depending on the chosen $\hat{c}>0$ such that for all $\Omega\in \mathcal{L}ip$ it holds $\mathrm{Vol}(\del 
\Omega)=\int_{\del \Omega} d\lambda \le M$.

The constant $M$ (and hence initially $\hat{c}$) can be chosen arbitrary large but finite.
We denote by $\Omega_0\in \mathcal{L}ip$ and $\Gamma_0\subset \overline{G}$ the ``reference'' domain and the ``reference'' boundary respectively (actually $\del \Omega_0=\Gamma_D\cup\Gamma_N \cup\Gamma_0$) corresponding to the initial shape before optimization.

Thus, the  admissible class of domains can be defined as
\begin{multline} \label{Uad}
 U_{ad}(\Omega_0,\eps,\hat{c},G)=\\\{\Omega\in\mathcal{L}ip  \, | \;
\Gamma_D\cup\Gamma_N\subset \del \Omega ,\; \Gamma \subset \overline{G}, \;
 M_0\le\int_\Gamma d\lambda \le M(\hat{c}),\; 
\int_\Omega \dx=\operatorname{Vol}(\Omega_0)\},
\end{multline}
 where $\hat{c}$ is given sufficiently large in the aim to have a sufficiently large constant $M>0$   in the sense that
it is not less than $M_0>0$, which is  the length of the straight line boundary. Moreover the case when $M$ is equal to the length of 
the plane boundary $M_0$ is the trivial case when $U_{ad}(\Omega_0,\eps,\hat{c},G)$ contains 
only 
one unique domain with the plane boundary, which hence is trivially optimal. 
Therefore the problem becomes interesting for a sufficiently large $M$.

In what follows we denote by $\lambda$ the $1$-dimensional Lebesgue 
measure on the Lipschitz boundary $\Gamma$, by $m_1$ the $1$-dimensional Hausdorff measure (which is equal to $\lambda$ on $\Gamma$) and we denote by
$u(\Omega,\mu)$  the weak solution of the Helmholtz problem on 
$\Omega$  satisfying~\eqref{eqVFHfh} with $1$-dimensional Radon measure $\mu$.

We define %
\begin{align} \label{Jen}
& J(\Omega,u(\Omega,\mu),\lambda)=A\int_\Omega 
|u(\Omega,\mu)|^2\dx+B\int_\Omega |\nabla u(\Omega,\mu)|^2 
\dx+C\int_\Gamma 
|u(\Omega,\mu)|^2 d\lambda
\end{align}
  for given $\mu$ and $\lambda$ and with   $A\ge 0$, $B\ge 0$,  $C\ge0$ positive constants for any  fixed $\omega>0$. 
  
  Ideally we would like to minimize $J(\Omega,u(\Omega,\lambda),\lambda)$ on 
  $U_{ad}(\Omega_0,\eps,\hat{c},G)$, however we are able to prove the existence of $\Omega_{opt}$ in $U_{ad}(\Omega_0,\eps,\hat{c},G)$ with  a 1-measure $\mu^*$, equivalent to $\lambda$, satisfying $\mu^*(\Gamma_{opt})\ge \lambda(\Gamma_{opt})$ on its boundary $\Gamma_{opt}$, such that $J(\Omega,u(\Omega,\mu^*),\lambda)$ realizes the infimum of $J(\Omega,u(\Omega,\lambda),\lambda)$ on 
  $U_{ad}(\Omega_0,\eps,\hat{c},G)$. So, if $\mu^*(\Gamma_{opt})=\lambda(\Gamma_{opt})$, $i.e.$ $\mu^*=\lambda$, then $\Omega_{opt}$ realizes the minimum of $J(\Omega,u(\Omega,\lambda),\lambda)$.

In order to keep the volume constraint, instead of Eq.~(\ref{Jen}) we can also consider the objective function 

\begin{align} \label{DefJ1}
 J_1(\Omega,u(\Omega,\lambda),\lambda)=&A\int_\Omega 
|u(\Omega,\lambda)|^2\dx+B\int_\Omega |\nabla u(\Omega,\lambda)|^2 
\dx+C\int_\Gamma |u(\Omega,\lambda)|^2 
d\lambda \nonumber\\
&+\kappa(\operatorname{Vol}(\Omega)-\operatorname{Vol}(\Omega_0))^2,
\end{align}

where $\kappa$ is some (large) positive constant penalizing the 
volume variation.

First of all let us prove the following lemma:
\begin{lemma}\label{LemmaUad}
	Let $U_{ad}(\Omega_0,\eps,\hat{c},G)$ be the class of admissible domains defined in~\eqref{Uad}.
	Then the following statements hold:
	\begin{enumerate}
		\item $U_{ad}(\Omega_0,\eps,\hat{c},G)$ is closed
	with respect to the Hausdorff convergence, in the sense of characteristic functions in $L^1(D)$ and in the sense of compacts.
	\item If $(\Omega_n)_{n\in \N^*}\subset U_{ad}(\Omega_0,\eps,\hat{c},G)$ then there exists a subsequence $(\Omega_{n_k})_{k\in \N^*}\subset (\Omega_n)_{n\in \N^*}$ and a domain $\Omega\in U_{ad}(\Omega_0,\eps,\hat{c},G)$ such that $(\Omega_{n_k})_{k\in \N^*}$ converges to $\Omega$ with respect to these three types of convergences and in addition $\overline{\Omega}_{n_k}$ and $\Gamma_{n_k}$  converge in the sense of Hausdorff respectively to $\overline{\Omega}$ and $\Gamma$.
	\item Let $(\Omega_n)_{n\in \N^*}\subset U_{ad}(\Omega_0,\eps,\hat{c},G)$ be a sequence  converging to $\Omega$ in $U_{ad}(\Omega_0,\eps,\hat{c},G)$ in the sense of point~2. Then there exists a subsequence $(\Omega_{n_k})_{k\in \N^*}$ with boundaries $(\Gamma_{n_k})_{k\in \N^*}$ and there exists a  positive Radon $1$-measure $\mu^*$ with support on $\Gamma$, equivalent to  the  one-dimensional Hausdorff measure on it or simply to its Lebesgue measure $\lambda$, such that
	\begin{enumerate}
		\item $\forall \psi \in C(\overline{D})$ $\quad\int_{\Gamma_{n_k}}\psi d \lambda\to \int_\Gamma \psi d \mu^*$,
		\item $\int_\Gamma d \lambda \le \int_\Gamma d \mu^*$ or equivalently $\mu^*(\Gamma)\ge \lambda(\Gamma)$.
	\end{enumerate}

	\end{enumerate}

\end{lemma}
\begin{proof}
\textbf{Point 1.}
	Let us denote by
$$U=\{\hbox{domains } \Omega\subset D|\; \Omega \hbox{ satisfies the } \eps-\hbox{cone property and }\Gamma_D\cup\Gamma_N\subset \del \Omega ,\; \Gamma \subset \overline{G}\},$$
which  obviously contains $U_{ad}(\Omega_0,\eps,\hat{c},G)$,
then thanks to~\cite{HENROT-2005} Theorem~2.4.10 p.~56 (see also p.~145), 
the set $U$  is  closed with respect to the Hausdorff convergence ($i.e.$ if $(\Omega_n)_{n\in \N^*}\subset U_{ad}(\Omega_0,\eps,\hat{c},G)$ and  $d_H(D\setminus \Omega_n, D\setminus \Omega)\to 0$ for $n\to +\infty$, which means that $\Omega_n \to \Omega$ in the sense of Hausdorff, then $\Omega\in U$), but also in the sense of characteristic functions, ($i.e.$ if $(\Omega_n)_{n\in \N^*}\subset U_{ad}(\Omega_0,\eps,\hat{c},G)$ and $\mathds{1}_{\Omega_n}\to \mathds{1}_{\Omega}$ for $n\to +\infty$ in $L^p_{loc}(\R^{2})$ for all $p\in [1,\infty[$, then $\Omega\in U$) and also in the sense of compacts (if for all  $K$ compact in $\Omega$ it follows that $K\subset \Omega_n$ and for all $O$ compact in $D\setminus \overline{\Omega}$ it follows that $O\in D\setminus \overline{\Omega}_n$  for a sufficiently large $n$, then $\Omega\in U$).
Consequently, if $(\Omega_n)_{n\in \N^*}\subset U_{ad}(\Omega_0,\eps,\hat{c},G)\subset U$ converges in these three senses to $\Omega \in U$, then, since for all $n$ the domains $\Omega_n$ are subsets of a fixed open set $D$ which is bounded in $\R^{2}$, the sequence $\mathds{1}_{\Omega_n}\to \mathds{1}_{\Omega}$ for $n\to +\infty$ in $L^1(D)$. This strong convergence of the characteristic functions~\cite[Prop.~2.3.6,~p.47]{HENROT-2005} implies that the perimeter and volume are respectively lower semicontinuous and continuous functions of $\Omega_n$. Hence, $$\operatorname{Vol}(\Omega)=\lim_{n\to +\infty} \operatorname{Vol}(\Omega_n)= \operatorname{Vol}(\Omega_0)$$
and in addition, since $\Gamma_N$ and $\Gamma_D$ are fixed and all $\Gamma_n\subset \overline{G}$,
\begin{equation}\label{EqLimIngGammaLambda}
	\int_{\Gamma} d \lambda\le \liminf_{n\to +\infty} \int_{\Gamma_n}  d \lambda.
\end{equation}
In particular~Eq.~(\ref{EqLimIngGammaLambda}) ensures that as soon as for all $\Gamma_n$, (\ref{EqSasha}) holds with $\hat{c}>0$ independent of $n$, 
then it holds also for $\Gamma$ and in the same way $\int_{\Gamma} d \lambda\le M$. As  $\Gamma$ is Lipschitz and in particular is given by a continuous curve, we also have $\int_{\Gamma} d \lambda\ge M_0$. Thus, we conclude that $U_{ad}(\Omega_0,\eps,\hat{c},G)$ is closed. 

\textbf{Point 2.} 
For the second statement it is sufficient to notice that Theorem~2.4.10 \cite[p.~56]{HENROT-2005}  holds on $U$. As $U$ is compact with respect to the  three considered convergences and as $U_{ad}(\Omega_0,\eps,\hat{c},G)$ is its closed subset, then it is compact too.

\textbf{Point 3.}
Let us prove the third point.
We take a sequence $(\Omega_n)_{n\in \N^*}$ in \\$U_{ad}(\Omega_0,\eps,\hat{c},G)$ which converges to $\Omega\in U_{ad}(\Omega_0,\eps,\hat{c},G)$ in all senses of point~2.
By Lemma~3.5~\cite{FALCONER-1985} the Lipschitz curves are $1$-sets (see Definition~\ref{Defdset}). We take the $1$-dimensional Hausdorff measures $\eta_n$ on $D$ in $\R^{2}$ having their supports on $\Gamma_n$. 
The sequence of measures $(\eta_n)_{n\in \N^*}\subset \mathcal{M}_b(D)$ is   uniformly bounded on $D$ (for all $n$ $\|\eta_n\|_1=\eta_n(D)=\int_{\Gamma_n} d \lambda\le M$). Then by~\cite[Thm.~1.59, p~26]{AMBROSIO-2000} or~\cite[Prop.~2.3.9,p.~48]{HENROT-2005}  there exits a weakly$^*$ convergent subsequence and there exists a Radon positive measure $\mu^*\in \mathcal{M}_b(D)$
such that $\eta_{n_k}$ converges weakly$^*$ to $\mu^*$
$$\forall\psi\in C_0(D) \quad \int_D \psi d\eta_{n_k}=\int_{\Gamma_{n_k}}\psi d \lambda\to \int_D \psi d \mu^*.$$
As $\Gamma_{n_k}$ converges in the sense of Hausdorff to $\Gamma$ then the support of $\mu^*$ is equal to~$\Gamma$:
$\int_D \psi d \mu^*= \int_\Gamma \psi d \mu^*,$ and hence
\begin{equation}\label{eqC0}
\forall \psi\in C_0(D)\quad \int_{\Gamma_{n_k}}\psi d \lambda\to  \int_\Gamma \psi d \mu^*.	
\end{equation}
But the restriction of  $\mu^*$  to $\Gamma$ 
is not necessarily equal to the 
Hausdorff measure on $\Gamma$, we only can prove that it is an equivalent $1$-dimensional measure to the usual Lebesgue measure on $\Gamma$ by~\cite[Prop.~1, p.~30]{JONSSON-1984} (see also Theorem~1 on p.~32). 
Indeed, taking in the definition of the weak$^*$ limit a particular $\psi$ equal to $1$ for all $x\in \overline{G}$ and knowing that all $\Gamma_{n_k}$ and $\Gamma$ itself are subsets of $\overline{G}$, we find
by the lower semicontinuity of perimeters that
$$\int_\Gamma  d \mu^*=\lim_k \int_{\Gamma_{n_k}}  d \lambda=\liminf_k \int_{\Gamma_{n_k}}  d \lambda\ge \int_\Gamma  d\lambda.$$
 For instance, if $\Gamma_n$ are  boundaries with a constant length which oscillate around a plane segment that has a length twice smaller than $\Gamma_n$,  and such that $\Gamma_n\to \Gamma$ in the sense of Hausdorff, it easy to see that 
  $$\lambda(\Gamma_n)\to \mu^*(\Gamma)=2 \lambda (\Gamma)> \lambda(\Gamma).$$ 
In addition, once again by the definition of $G$ containing all $\Gamma_n$ and $\Gamma$ in $D$, the weak$^*$ limit~\eqref{eqC0} also holds for all $\psi\in C(\overline{D})$.
Consequently, if $\eta$ is the $1$-dimensional Hausdorff measure with support equal to $\Gamma$, then $\mu^* \ge \eta$  and thus there exists $c_1>0$ such that 
$$c_1r\le \eta(B(x,r))\le \mu^*(B(x,r)) \quad x\in \Gamma, 0<r\le1.$$

To show that $\mu^*$ is also a $1$-measure we  need to have the upper bound too.
Since $B(x,r)$ is open, it holds, thanks to~(\ref{EqSasha}) and to the uniform bound of lengths $M_0\le \eta_{n_k}(\Gamma_{n_k})=\lambda(\Gamma_{n_k})\le M$, 
$$\frac{1}{\mu^*(\Gamma)}\mu^*(B(x,r))\leq \liminf_k \frac{1}{\int_{\Gamma_{n_k}}d \lambda}\eta_{n_k}(B(x,r))\leq \frac{1}{M_0}\liminf_k \eta_{n_k}(B(x,r))\le \frac{1}{M_0}\hat{c}r$$ 
by the Portmanteau theorem (see for instance Theorem 11.1.1 in Dudley~\cite{DUDLEY-1989}).  Hence, $\mu^*(B(x,r))\leq c_2 r$ with the constant $c_2=\frac{\mu^*(\Gamma)}{M_0}\hat{c}>0$ independent of $n_k$.

Then we conclude that $\mu^*$ is $1$-measure on $\Gamma$ and thus equivalent to $\eta$ by~\cite[Prop.~1,p.~30]{JONSSON-1984}.
\end{proof}
Let us now prove  the existence of an optimal shape in a certain sense:
\begin{theorem}\label{ThPrincipale}
Let $\Omega_0\subset D$ be a domain of the class $\mathcal{L}ip$ with a  Lipschitz boundary $\del \Omega_0$ of  bounded length, such that $\Gamma_D\cup\Gamma_N\subset \del \Omega_0$ and $\Gamma_0=\del\Omega_0\setminus(\Gamma_D\cup\Gamma_N) \subset \overline{G}$, $U_{ad}$ be defined by~(\ref{Uad}) and $\omega>0$ be fixed.
For the objective function $J(\Omega, u(\Omega,\lambda),\lambda)$,  defined in~(\ref{Jen}) and 
constructed with the weak solution of the Helmholtz problem~(\ref{Helmholtz}) for some fixed $\alpha\in C(\overline{G}),$ $f\in L_2(D)$, $g\in H^\frac{1}{2}(\Gamma_D)$ and $h\in H^1(D)$, 
there exists  $\Omega_{opt}\in U_{ad}(\Omega_0,\eps,\hat{c},G)$ and there exists a finite valued $1$-dimensional positive measure $\mu^*$ on its boundary $\Gamma_{opt}$ equivalent to $\lambda$ such that 
\begin{equation}\label{EqMu*LamdaGamma}
    \int_{\Gamma_{opt}}d \mu^* \ge \int_{\Gamma_{opt}}d \lambda                                                                                                                                                                                                                                                                                                                                              \end{equation}
 and
\begin{equation}
	J(\Omega_{opt}, u(\Omega_{opt}, \mu^*),\lambda)\le\inf_{\Omega\in 
U_{ad}(\Omega_0,\eps,\hat{c},G)}J(\Omega, u(\Omega,\lambda),\lambda)= J(\Omega_{opt}, u(\Omega_{opt}, \mu^*),\mu^*).
\end{equation}
If $\mu^*=\lambda$, then $J(\Omega_{opt}, u(\Omega_{opt}, \lambda),\lambda)$ is the minimum on $U_{ad}(\Omega_0,\eps,\hat{c},G)$.

\end{theorem}

\begin{proof}
 Let us start by noticing that, since all $\Omega \in U_{ad}(\Omega_0,\eps,\hat{c},G)$ are included in the same domain $D$, it follows that the Poincar\'e constants in~Theorem~\ref{ThWPc} can all be bounded by the same constant $C(\operatorname{Vol}(D))$ depending only on the volume of $D$.

 Let $(\Omega_n)_{n\in \N^*}\subset U_{ad}(\Omega_0,\eps,\hat{c},G)$ be a minimizing sequence of the functional $J(\Omega)$ (this minimizing sequence exists since $J(\Omega)\ge 0$).  
 
 Hence, by~Lemma~\ref{LemmaUad}  
 there exist  a domain $\Omega$ with $\Gamma$
  and a subsequence $(\Omega_{n_k})_{k\in \N}$ converging in these three senses to $\Omega$ and such that $\overline{\Omega}_{n_k}$ and $\Gamma_{n_k}$ converge in the sense of Hausdorff to $\overline{\Omega}$ and $\Gamma$ respectively (all other boundary parts are the same as the sequence is in $U_{ad}(\Omega_0,\eps,\hat{c},G)$).  
 In the aim to abbreviate the notations, in what follows the index $n_k$ is changed to $n$. 
 
 By point~3 of~Lemma~\ref{LemmaUad}  there exist a subsequence of boundaries $(\Gamma_{n_k})_{k\in \N^*}\subset (\Gamma_n)_{n\in \N^*}$ and a positive measure $\mu^*$ on $\Gamma$, equivalent to $\lambda$, such that~Eq.~(\ref{EqMu*LamdaGamma}) holds and 
 \begin{equation}\label{EqPassInt}
 	\forall \psi\in C(\overline{D}) \quad 
 \int_{\Gamma_{n_k}} \psi d \lambda\to \int_{\Gamma} \psi d \mu^*.
 \end{equation}
 Therefore, all norms of $L^2(\Gamma)$ and $V(\Omega)$ computed with this measure on $\Gamma$ are equivalent to the corresponding norms computed with the Lebesgue measure on $\Gamma$.
 
To avoid complicated notations we denote again $\Gamma_{n_k}$ by $\Gamma_n$. 

 Let us consider now the solutions $(u_n)_{n\in \N^*}$ of the Helmholtz problem 
on $(\Omega_n)_{n\in \N^*}$.

Since domains of $U_{ad}(\Omega_0,\eps,\hat{c},G)$) have Lipschitz boundaries with finite perimeters,
we use (see~\cite{CHENAIS-1975}) the fact that
the norm of the extension operator $E: H^1(\Omega)\to H^1(\R^{2})$ 
is bounded on such a family of domains 
$i.e$ there exists a constant $C_E>0$ independent of $n$ such that
\begin{equation}
 \|E u_n |_{D}\|_{H^1(D)}\le C_E\|u_n\|_{H^1(\Omega_n)},
\end{equation}
to deduce that the sequence $(E u_n |_{D})_{n\in \N^*}$ is bounded in $H^1(D)$:
 \begin{align*}
  \|Eu_n|_D \|_{H^1(D)} & \le C_E\|u_n\|_{H^1(\Omega_n)} \\
  & \le C_E C(\alpha,\omega, 
C_p(\Omega_n))\left(\|f\|_{L_2(\Omega_n)}+\|g\|_{H^\frac{1}{2}(\Gamma_D)}+\|h\|_
{H^1(\Omega_n)}\right)\\
  & \le C_E 
C(\alpha,\omega,\mathrm{Vol}(D))\left(\|f\|_{L_2(D)}+\|g\|_{H^\frac{1}{2}
(\Gamma_D)}+C_E\|h\|_{H^1(D)}\right),
 \end{align*}
 which means that there exists a constant $C>0$ independent of $n$ such that 
 $$\|Eu_n|_D \|_{H^1(D)}\le C \hbox{ for all } n.$$
 Here, in addition to the uniform boundedness of the extension operators, we 
have also used~Eq.~(\ref{EqApriori}) and the fact that $\Gamma_D$ is the same for 
all $n$.
  Consequently, there exists $u^*\in H^1(D)$ such that $E u_n 
|_{D}\rightharpoonup u^*$ in $H^1(D)$.
 By compactness of the trace operator $\operatorname{Tr}_\Gamma: H^1(D)\to 
L_2(\Gamma)$  and of the inclusion of $H^1(D)$ in $L_2(D)$ 
(see~Theorem~\ref{ThGENERIC}), we directly have that 
$\operatorname{Tr}_\Gamma(Eu_n|_{D})\to \operatorname{Tr}_\Gamma(u^*)$ in 
$L_2(\Gamma)$ and $E u_n |_{D}\to u^*$ in $L_2(D)$.
 
 Let us show that $u^*$ is equal to the weak solution $u(\Omega,\mu^*)$ of~(\ref{eqVFHfh}) on 
$\Omega$ in the sense of the variational formulation~(\ref{eqVFHfh})  considered with $\mu^*$.
 
 From the variational formulation~(\ref{eqVFHfh}), taking $f\in L_2(D)$  and 
$h\in H^1(D)$, let us consider linear functionals defined for a fixed $v\in V(D)$ 
and for all $w_n$ and $w$ in $V(D)$
 \begin{align*}
  F^n[w_n,v]=&(\nabla w_n,\nabla 
v)_{L_2(\Omega_n)}-\omega^2(w_n,v)_{L_2(\Omega_n)}+(\alpha 
w_n,v)_{L_2(\Gamma_n, \lambda)}\\
  &+(f,v)_{L_2(\Omega_n)}-(\operatorname{Tr}_{\Gamma_n} 
h,v)_{L_2(\Gamma_n, \lambda)},\\
  F[w,v]=&(\nabla w,\nabla 
v)_{L_2(\Omega)}-\omega^2(w,v)_{L_2(\Omega)}+(\alpha 
w,v)_{L_2(\Gamma, \mu^*)}\\
  &+(f,v)_{L_2(\Omega)}-(\operatorname{Tr}_{\Gamma} h,v)_{L_2(\Gamma, 
\mu^*)}.
 \end{align*}

We start by showing that as soon as $w_n\rightharpoonup w$ in $V(D)$
\begin{equation}\label{EqFconv}
 \forall v\in  V(D)\quad  F^n[w_n,v]\to F[w,v] \hbox{ for } n\to +\infty.
\end{equation}
Thus we consider
\begin{align*}
 \left|F^n[w_n,v]-F[w,v]\right|&\le|(\nabla w_n,\mathds{1}_{\Omega_n}\nabla 
v)_{L_2(D)}- 
(\nabla w,\mathds{1}_{\Omega}\nabla v)_{L_2(D)}|\\
+&\omega^2|(w_n,\mathds{1}_{\Omega_n} v)_{L_2(D)}-(w,\mathds{1}_{\Omega} 
v)_{L_2(D)}|\\ +&\left|(\alpha  w_n, v)_{L_2(\Gamma_n, \lambda)}-    
(\alpha w, v)_{L_2(\Gamma, \mu^*)}\right|
+|(f,(\mathds{1}_{\Omega_n} -\mathds{1}_{\Omega}) v)_{L_2(D)}|\\
+&|(\operatorname{Tr}_{\Gamma_n} 
h,v)_{L_2(\Gamma_n,\lambda)}-(\operatorname{Tr}_{\Gamma} 
h,v)_{L_2(\Gamma,\mu^*)}|.
\end{align*}

Since $\Omega_n\to\Omega$ in the sense of characteristic functions and $v\in 
H^1(D)$, we directly have that
$\mathds{1}_{\Omega_n}\nabla v\to \mathds{1}_{\Omega}\nabla v$ in $L_2(D)$, 
which with $w_n\rightharpoonup w$ in $V(D)$ gives that 
$$(\nabla w_n,\mathds{1}_{\Omega_n}\nabla v)_{L_2(D)}\to
(\nabla w,\mathds{1}_{\Omega}\nabla v)_{L_2(D)} \hbox{ for } n\to +\infty.$$
By the compactness of  the inclusion of $H^1(D)$ in $L_2(D)$, $w_n \to w$ in 
$L_2(D)$ and by the convergence of the characteristic functions
$\mathds{1}_{\Omega_n} v\to \mathds{1}_{\Omega} v$ in $L_2(D)$, hence we also 
have
$$(w_n,\mathds{1}_{\Omega_n} v)_{L_2(D)}\to (w,\mathds{1}_{\Omega} 
v)_{L_2(D)}$$ and similarly, $(f,(\mathds{1}_{\Omega_n} -\mathds{1}_{\Omega}) 
v)_{L_2(D)}\to 0$.
 
 Let us prove that
 \begin{equation}\label{EqTrGoalW}
 \forall v\in C(\overline{D})\cap V(D) \quad (\alpha  w_n, 
v)_{L_2(\Gamma_n,\lambda)}\to  (\alpha w, v)_{L_2(\Gamma,\mu^*)}.
 \end{equation}

 Thanks to~\cite{MAZ'JA-1985} Theorem~1.1.6/2, for all domains $\Omega \in 
U_{ad}(\Omega_0,\eps,\hat{c},G)$ and $D$ itself, the space $C^\infty(\overline{\Omega})\cap 
V(\Omega)$ is dense in $V(\Omega)$. Thus   there exists a sequence 
$(\phi_m)_{m\in \N} \subset C^\infty(\overline{D})\cap V(D)$ converging 
strongly to $w\in V(D)$. 
Therefore, following~\cite{CAPITANELLI-2010} we have
\begin{align}
 &\left|\int_{\Gamma_n}\mathrm{Tr} w_n\mathrm{Tr} v d \lambda 
-\int_{\Gamma}\mathrm{Tr} w\mathrm{Tr} v 
d \mu^*\right|\le\nonumber\\
 &\le \left|\int_{\Gamma_n}|\mathrm{Tr} w_n-\mathrm{Tr} 
w||\mathrm{Tr} v| d \lambda \right|+\left|\int_{\Gamma_n}|\mathrm{Tr} 
w-\mathrm{Tr} \phi_m||\mathrm{Tr} v| d \lambda\right| \nonumber\\
 &+\left|\int_{\Gamma_n}\mathrm{Tr} \phi_m\mathrm{Tr} v d \lambda 
-\int_{\Gamma}\mathrm{Tr} \phi_m\mathrm{Tr} v 
d \mu^*\right|+\left|\int_{\Gamma}|\mathrm{Tr} \phi_m-\mathrm{Tr} 
w||\mathrm{Tr} v| d\mu^* \right|.\label{EqBigEstInt}
\end{align}
We start by  estimating the first term in~Eq.~(\ref{EqBigEstInt}): 
control it with the Cauchy-Schwarz inequality 
\begin{align*}
 &\left|\int_{\Gamma_n}|\mathrm{Tr} w_n-\mathrm{Tr} w||\mathrm{Tr} 
v|d \lambda\right|\le \|\mathrm{Tr} (w_n - 
w)\|_{L_2(\Gamma_n,\lambda)}\|\mathrm{Tr}v\|_{L_2(\Gamma_n,\lambda)}.
 \end{align*}
Moreover, %
 there exists a positive constant $C_\sigma>0$ independent of $n$ such that for $\frac{1}{2}<\sigma \le 1$ it holds 
 \begin{equation}\label{EqTrsigma}
 \forall w\in H^\sigma(\R^{2}) \quad \|\mathrm{Tr}_{\Gamma_n} 
w\|_{L_2(\Gamma_n, \lambda)}^2\le C_\sigma \|w\|_{H^\sigma(\R^{2})}^2.
 \end{equation}
It is a direct corollary of the proof of~\cite{CAPITANELLI-2010} Theorem~5.3 and 
the fact that the lengths of $\Gamma$ and all $\Gamma_n$ are finite and bounded 
by a constant, denoted by $M$. In addition, 
$$\|\mathrm{Tr}v\|_{L_2(\Gamma_n,\lambda)}\le \lambda(\Gamma_n)\|v\|_{L^\infty(D)}\le M\|v\|_{L^\infty(D)}.$$
  Moreover, by Theorem~5.8~\cite{CAPITANELLI-2010},  for $D$ (but also for all domains in $U_{ad}(\Omega_0,\eps,\hat{c},G)$) there exists a bounded linear extension operator $E_{\sigma}:H^{\sigma}(D)\rightarrow H^{\sigma}(\mathbb{R}^2)$, $\frac{1}{2}<\sigma\leq 1$, with
\begin{equation}\label{EqExtsigma}
\Vert E_\sigma v\Vert_{H^{\sigma}(\mathbb{R}^2)}\leq C_{D} \Vert v\Vert_{H^{\sigma}(D)}.
\end{equation}
 Hence, applying~Eq.~(\ref{EqTrsigma}) and~Eq.~(\ref{EqExtsigma}), we obtain that
 \begin{equation*}
  \|\mathrm{Tr} (w_n - w)\|_{L_2(\Gamma_n,\lambda)}\le 
C_\sigma\|E_\sigma(w_n-w)\|_{H^\sigma(\R^{2})}\le C_\sigma 
C_{D}\|w_n-w\|_{H^{\sigma}(D)},
 \end{equation*}
from where, by the compactness of the embedding of $H^1(D)$ in $H^{\sigma}(D)$ for $\frac{1}{2}<\sigma<1$, we finally have that $\|w_n-w\|_{H^{\sigma}(D)}\to 0$ for $n\to +\infty$ and consequently the first term in~Eq.~(\ref{EqBigEstInt}) converges to $0$ for $n\to +\infty$. %
 
 For the second term (and in the same way the last term)  in~Eq.~(\ref{EqBigEstInt}),  as previously %
 we directly find
 $$\left|\int_{\Gamma_n}|\mathrm{Tr} w-\mathrm{Tr} \phi_m||\mathrm{Tr} v| 
d \lambda\right|
 \le C \|w-\phi_m\|_{H^1(D)}\to 0 \hbox{ for } m\to +\infty$$
 with a constant $C>0$ independent of $n$.
 For the last term we simply replace $\Gamma_n$ by~$\Gamma$, knowing that $\Omega\in U_{ad}(\Omega_0,\eps,\hat{c},G)$.
 Hence,  for all $\eps>0$ there exists $k\in \N$ (uniformly on $n$) such that 
 \begin{equation*}\label{Eqest24}
  \forall m\ge k \quad \max\left\{\left|\int_{\Gamma_n}|\mathrm{Tr} w 
-\mathrm{Tr} \phi_m||\mathrm{Tr} v| d\lambda\right|, 
\left|\int_{\Gamma}|\mathrm{Tr} \phi_m -\mathrm{Tr} w||\mathrm{Tr} v| 
d\mu^*\right|\right\}<\eps.
 \end{equation*}
Thus, let us fix  such an $m$.
 
 Finally, for the third term in~Eq.~(\ref{EqBigEstInt}),  we use~Eq.~(\ref{EqPassInt})  which, by the continuity and the boundedness of $\phi_m v$ in $\overline{D}$ with a standard density argument, implies 
 \begin{equation}\label{EqwmTrInf}
  \int_{\Gamma_n}|\mathrm{Tr} \phi_m\mathrm{Tr} v|d \lambda \to 
\int_{\Gamma}|\mathrm{Tr} \phi_m\mathrm{Tr} v|d \mu^* \hbox{ for } n\to +\infty.
 \end{equation}

 Therefore, for the sufficiently large $m$ that we have fixed, we also have  that 
 \begin{equation*}\label{Eqest3}
  \forall \eps>0\; \exists p\in \N: \; \forall n \ge p \; 
\left|\int_{\Gamma_n}|\mathrm{Tr} \phi_m\mathrm{Tr} v| d \lambda 
-\int_{\Gamma}|\mathrm{Tr} \phi_m\mathrm{Tr} v| d \mu^*\right|<\eps.
 \end{equation*}
Putting  all results together for the four terms of Eq.~(\ref{EqBigEstInt}),
we obtain~Eq.~(\ref{EqTrGoalW}), which by the density of $C(\overline{D})\cap V(D)$ in $V(D)$, also holds for all $v\in V(D)$.
Consequently, we also have, as $h\in V(D)$
$$\forall v\in V(D) \quad (\mathrm{Tr}_{\Gamma_n} h, v)_{L_2(\Gamma_n,
\lambda)}\to (\mathrm{Tr}_{\Gamma} h, v)_{L_2(\Gamma, \mu^*)} \quad \hbox{for } 
n\to +\infty.$$
This concludes the proof of~Eq.~(\ref{EqFconv}).

 Hence, taking $w_n=Eu_n|_D\in V(D)$, $i.e.$ the extensions of solutions on $\Omega_n$, which are uniformly bounded and  weakly converge to 
 $u^*\in V(D)$,
 we find that  for all $v\in V(D)$
 $$0=F^n[Eu_n|_D,v]\to F[u^*,v]=0 \hbox{ for } n \to +\infty.$$
 This means that $u^*$ is a weak solution on $\Omega$  in the sense of the variational formulation~(\ref{eqVFHfh}) considered with $\mu^*$, and by the uniqueness of the weak solution on $\Omega$, $u^*|_\Omega=u(\Omega,\mu^*)$.
  In order to conclude that the infimum of $J$ is realized, we shall prove that 
  $$\inf_{\Omega\in U_{ad}(\Omega_0,\eps,\hat{c},G)} J(\Omega,u(\Omega,\lambda),\lambda)=\lim_{n\to +\infty}J(\Omega_n,u_n(\Omega_n,\lambda),\lambda)=J(\Omega,u(\Omega,\mu^*),\mu^*).$$
 Let us start by showing that 
 $(\mathds{1}_{\Omega_n}Eu_n|_D)_{n\in \N^*}$ converges strongly to $\mathds{1}_{\Omega}Eu|_D$ in $V(D)$.
 Firstly, we find in the same way as previously (see~\cite{CAPITANELLI-2010-1}) that 
 \begin{equation}\label{EqTrGoal}
  \left|\int_{\Gamma_n}|\mathrm{Tr} w_n|^2 d \lambda -\int_{\Gamma}|\mathrm{Tr} 
w|^2 d \mu^*\right|\to 0 \quad \hbox{for } n\to +\infty.
 \end{equation}
 Then, once again, by the fact that the weak convergence of $(Eu_n|_D)_{n\in \N^*}$ to $Eu|_D$ in $V(D)$ implies the strong convergence in $L_2(D)$ 
 and by~Eqs.~(\ref{EqTrGoalW}),~(\ref{EqTrGoal}), we obtain
 \begin{align*}
  &\lim_{n\to +\infty}\|\mathds{1}_{\Omega_n}Eu_n|_D\|_{V(D,\lambda)}^2=
  \lim_{n\to +\infty}\left(\omega^2 
\|\mathds{1}_{\Omega_n}Eu_n|_D\|^2_{L_2(D)}-i\int_{\Gamma_n} \operatorname{Im} 
\alpha|\operatorname{Tr} u_n|^2d \lambda\right.
  \\
  &\left.-\int_D \mathds{1}_{\Omega_n}f\overline{Eu_n|_D}\dx+\int_{\Gamma_n} 
\operatorname{Tr} h \overline{\operatorname{Tr} u_n}d \lambda\right) =\omega^2 
\|\mathds{1}_{\Omega}Eu|_D\|^2_{L_2(D)}\\
  &-i\int_{\Gamma} \operatorname{Im} \alpha|\operatorname{Tr} u|^2d \mu^*-\int_D 
\mathds{1}_{\Omega}f\overline{Eu|_D}\dx+\int_{\Gamma} \operatorname{Tr} h 
\overline{\operatorname{Tr} u}d \mu^*
  =\|u\|^2_{V(\Omega,\mu^*)}.
 \end{align*}
 Since we have at the same time the weak convergence and the convergence of norms, it implies the strong convergence in $V(D)$
 of $(\mathds{1}_{\Omega_n}Eu_n|_D)_{n\in \N^*}$  to $\mathds{1}_{\Omega}Eu|_D$.

 Hence, as the functional $J$, which can be considered as an equivalent norm on $V(D)$, is continuous:
 $$\lim_{n\to +\infty}J(\Omega_n,u_n(\Omega_n,\lambda),\lambda)=J(\Omega,u(\Omega,\mu^*),\mu^*),$$
 $i.e.$, as $\Omega_n$ is a minimizing sequence of $J$, $J(\Omega,u(\Omega,\mu^*),\mu^*)$ is the infimum for all $\Omega\in U_{ad}(\Omega_0,\eps,\hat{c},G)$. 
 
 By the relation $\int_\Gamma d\mu^*\ge \int_\Gamma d \lambda$ we directly have 
 \begin{equation}\label{EqInfEQ}
 	J(\Omega,u(\Omega,\mu^*),\mu^*)\ge J(\Omega,u(\Omega,\mu^*),\lambda).
 \end{equation}
 If we have $\mu^* =\lambda$, then 
 $J(\Omega,u(\Omega,\lambda),\lambda)$ is the minimum.
 \end{proof}

\section{Conclusion}
Started by the well-posedness result on a large class of domains with $d$-set boundaries including even fractal boundaries, we showed that the problem of finding an optimal shape for the Helmholtz problem with a dissipative boundary has at least one solution in the sense of a suitable measure $\mu^*$. 

\section*{Acknowledgments} We would like to thank for fruitful discussions and for a strong support of our work Profs. M.R. Lancia, M. Hinz and especially A. Teplyaev for his helpful remarks and advice, crucial in the writing of the article. We also thank our anonymous referee for pointing the interest and main difficulties in the developing of our shape optimal existence theory.


\appendix
\section{$d$-sets and trace theorems on a $d$-set}\label{AnnexB}
Let us define the main notions which we use in~Theorem~\ref{ThWPc}. 
 \begin{definition}[Ahlfors $d$-regular set or $d$-set~\cite{JONSSON-1984}]\label{Defdset} 
 Let $F$ be a closed Borel non-empty subset of $\R^N$. The set $F$ is is called a $d$-set ($0<d\le N$) if there exists a $d$-measure  $\mu$ on $F$, $i.e.$ a positive Borel measure with support $F$ ($\operatorname{supp} \mu=F$) such that there exist constants 
$c_1$, $c_2>0$,
\begin{equation*}
 c_1r^d\le \mu(\overline{B(x,r)})\le c_2 r^d, \quad \hbox{ for  } ~ \forall~x\in F,\; 0<r\le 1,
 \end{equation*}
where $B(x,r)\subset \R^N$ denotes the Euclidean ball centered at $x$ and of radius~$r$.
\end{definition}
As~\cite[Prop.~1, p~30]{JONSSON-1984} all $d$-measures on a fixed $d$-set $F$ are equivalent, it is also possible to define a $d$-set by the $d$-dimensional Hausdorff measure $m_d$:
 \begin{equation*}
 c_1r^d\le m_d(F\cap \overline{B(x,r)})\le c_2 r^d, \quad \hbox{ for  } ~ \forall~x\in F,\; 0<r\le 1
 \end{equation*}
 which in particular implies that $F$ has Hausdorff dimension $d$ in the neighborhood of each point of $F$~\cite[p.33]{JONSSON-1984}.

If the boundary $\del \Omega$ is a $d$-set endowed with a $d$-measure $\mu$, then we denote by $L_2(\del \Omega, \mu)$ the Lebesgue space defined with respect to this measure with the norm
$$\|u\|_{L_2(\del \Omega,\mu)}=\left(\int_{\del \Omega} |u|^2 d \mu \right)^\frac{1}{2} .$$

  In particular,  $N$-sets ($d$-set with $d=N$) satisfy
 $$\exists c>0\quad \forall x\in \overline{\Omega}, \; \forall r\in]0,\delta[\cap]0,1] \quad \lambda(B(x,r)\cap \Omega)\ge C \lambda(B(x,r))=cr^N,$$
 where $\lambda(A)$ denotes the Lebesgue measure of a set $A$ of $\R^N$. This property is also called the measure density condition~\cite{HAJLASZ-2008}. Let us notice that an $N$-set 
$\Omega$ cannot be ``thin'' close to its boundary $\del \Omega$.

The trace operator on a $d$-set is understood in the following way:
 \begin{definition}[Trace operator]\label{DefGTrace}
  For an arbitrary open set $\Omega$ of $\R^N$, the trace operator $\mathrm{Tr}$ is defined~\cite{JONSSON-1984} for $u\in L^1_{loc}(\Omega)$ by
 $$
  \mathrm{Tr} \, u(x)=\lim_{r\to 0} \frac{1}{\lambda(\Omega\cap B(x,r))}\int_{\Omega\cap B(x,r)}u(y)\dy,
 $$
 where $\lambda$ denotes the Lebesgue measure.
 The trace operator $\mathrm{Tr}$ is considered for all $x\in\overline{\Omega}$ for which the limit exists.
 \end{definition}
Hence, the following Theorem (see Ref.~\cite{ARFI-2017} Section~2) 
generalizes the classical results~\cite{LIONS-1972,MARSCHALL-1987} for domains with the Lipschitz boundaries $\del \Omega$:
\begin{theorem}\label{ThGENERIC}
Let $\Omega$ be an admissible domain in $\R^N$ in the sense of Ref.~\cite{ARFI-2017}, $i.e.$ 
 $\Omega$ is an $N$-set, such that its boundary $\del \Omega$ is a compact $d$-set with a $d$-measure $\mu$, $N-2< d<N$, and 
  the norms $\|f\|_{H^1(\Omega)}$ and $\|f\|_{C_2^1(\Omega)}=\|f\|_{L_2(\Omega)}+\|f_{1,\Omega}^\sharp\|_{L_2(\Omega)}$ with
  $$f_{1,\Omega}^\sharp(x)=\sup_{r>0} r^{-1}\inf_{c\in 
\R}\frac{1}{\lambda(B(x,r))}\int_{B(x,r)\cap \Omega}|f(y)-c|\dy,$$
  where $\lambda$ is the $N$-dimensional Lebesgue measure,  are equivalent on 
$H^1(\Omega)$.
  Then, \begin{enumerate}
       \item $H^1(\Omega)$ is compactly embedded in $L_2^{loc}(\Omega)$ or in $L_2(\Omega)$ if $\Omega$ is bounded;
       \item $\mathrm{Tr}_\Omega: H^1(\R^N)\to H^1(\Omega)$ is a linear continuous and surjective operator with linear bounded inverse (the extension operator $E_\Omega: H^1(\Omega)\to H^1(\R^N)$);
       \item for $\beta=1-(N-d)/2>0$ the operators $\mathrm{Tr}: H^1(\R^N)\to L_2(\del \Omega),$ and $\mathrm{Tr}_{\del \Omega}:H^1(\Omega)\to L_2(\del \Omega)$  are linear compact operators with dense image
       $\operatorname{Im}(\mathrm{Tr})=\operatorname{Im}(\mathrm{Tr}_{\del \Omega})= B^{2,2}_{\beta}(\del \Omega)$ and with linear bounded  right inverse (the extension operators) $E: B^{2,2}_{\beta}(\del \Omega)\to H^1(\R^N)$ and $E_{\del \Omega}: B^{2,2}_{\beta}(\del \Omega)\to H^1(\Omega);$
       \item the Green formula holds for all $u$ and $v$ from $H^1(\Omega)$ with $\Delta u\in L_2(\Omega)$:
\begin{equation}\label{FracGreen}
 \int_\Omega v\Delta u\dx + \int_\Omega \nabla v. \nabla u \dx=\langle \frac{\del u}{\del \nu}, 
\mathrm{Tr}v\rangle _{((B^{2,2}_{\beta}(\del \Omega))', B^{2,2}_{\beta}(\del \Omega))},
\end{equation}
where the dual Besov space $(B^{2,2}_{\beta}(\del \Omega))'=B^{2,2}_{-\beta}(\del \Omega)$ is introduced in Ref.~\cite{JONSSON-1995}.
\item the usual integration by parts holds for all $u$ and $v$ from $H^1(\Omega)$ in the following weak sense
      \begin{equation}\label{IPP}
 \langle u \nu_i,v\rangle_{(B^{2,2}_{\beta}(\del \Omega))', B^{2,2}_{\beta}(\del \Omega))}:= \int_\Omega \frac{\del u}{\del x_i} v\dx+\int_{\Omega} u\frac{\del v}{\del x_i}\dx \quad i=1,\ldots,N,
\end{equation}
where by $u \nu_i$ is denoted the linear continuous functional on $B^{2,2}_{\beta}(\del \Omega)$.
\item $\|u\|_{H^1(\Omega)}$ is equivalent to $\|u\|_{\mathrm{Tr}}=\left( \int_\Omega |\nabla u|^2\dx +\int_{\del \Omega} |\mathrm{Tr}u|^2 d\mu \right)^\frac{1}{2}. $
      \end{enumerate}
      
\end{theorem}
  Theorem~\ref{ThGENERIC} is a particular case of the results proven in~Ref.~\cite{ARFI-2017}. We also notice that in the framework of the Sobolev space $H^1$ and the Besov spaces $B^{2,2}_{\beta}$ with $\beta<1$, as here, we do not need to impose  Markov's local inequality on $\del \Omega$, as it is trivially satisfied (see Ref.~\cite{JONSSON-1997} p. 198). 
  To prove formula~\eqref{IPP} we follow the proof of formula~(4.11) of Theorem~4.5 in~\cite{ARXIV-CREO-2018} using the existence of a sequence of domains $(\Omega_m)_{m\in \N^*}$ with Lipschitz boundaries such that $\Omega_m\subset \Omega_{m+1}$ and $\Omega=\cup_{m=1}^\infty \Omega_m$.
  

\section{Approximation of the damping parameter $\alpha$ in the Robin boundary condition by a model with dissipation in the volume}\label{Annex}

\begin{theorem}\label{ThPascal} 
Let $\Omega= ]-L,L[ \; \times \; ]-\ell,\ell[$  be a domain with a  simply connected sub-domain $\Omega_0$, whose boundaries are $]-L,0[ \; \times \{ \ell \}$, $\{ -L \} \; \times \; ]-\ell,\ell[ $, $]-L,0[ \; \times \{ -\ell \}$ and another boundary, denoted by $\Gamma$, which is the straight line starting in $(0,-\ell)$ and ending in $(0,\ell)$.
In addition let $\Omega_1$ be the supplementary domain of $\Omega_0$ in $\Omega$, so that $\Gamma$ is the common boundary of $\Omega_0$ and $\Omega_1$. The length $L$ is supposed to be large enough.

Let the original problem (the frequency version of the wave damped problem \eqref{amortih})~be 
\begin{align}
-\nabla \cdot (\eta_0 \nabla u_0) -\omega^2 \xi_0 u_0 & =  0 \; \mbox{ in } \Omega_0,\label{eq:Ap1}\\
-\nabla \cdot (\eta_1 \nabla u_1) -\omega^2 \tilde{\xi}_1 u_1 & =  0 \; \mbox{ in } \Omega_1,\label{eq:Ap2}
\end{align}
with 
$$
\tilde{\xi}_1 = \xi_1 \left( 1 + \frac{ai}{\xi_1\omega}\right),
$$
together with boundary conditions on $\Gamma$
\begin{equation} \label{eq:continuity_conditions}
u_0=u_1 \; \; \mbox{and} \; \; 
\eta_0 \nabla u_0 \cdot n = \eta_1 \nabla u_1 \cdot n,
\end{equation}
and the condition on the left boundary
\begin{equation} \label{eq:left_bc}
u_0(-L,y)=g(y),
\end{equation}
and some other boundary conditions.
Let the modified problem be 
\begin{equation} \label{eq:equation_u3}
-\nabla \cdot (\eta_0 \nabla u_2) -\omega^2 \xi_0 u_2 = 0 \; \mbox{ in } \Omega_0 \,
\end{equation}
with boundary absorption condition on $\Gamma$
\begin{equation}\label{eq:cl_robin_gamma}
\eta_0 \nabla u_2 \cdot n + \alpha u_2 =0
\end{equation}
and the condition on the left boundary
\begin{equation}\label{eq:cl_dirichlet_u3_left}
u_2(-L,y)=g(y).
\end{equation}
Let $u_0$, $u_1$, $u_2$ and $g$ be decomposed into Fourier modes in the $y$ direction, denoting by $k$ the associated wave number.
Then the complex parameter $\alpha$, minimizing the following expression
$$
A||u_0-u_2||^2_{L_2(\Omega_0)} + B ||\nabla(u_0-u_2)||^2_{L_2(\Omega_0)}
$$
can be found from the minimization of the error function
$$
e(\alpha):= \sum_{k=\frac{n\pi}{L}, n \in \mathbb{Z}} e_k (\alpha),
$$
where $e_k$ are given by
\begin{multline}
 e_k (\alpha) = (A+B|k|^2) \left( \frac{1}{2\lambda_0} \left\{  |\chi|^2  \left[1 - \exp(-2\lambda_0 L) \right] 
                          \right. \right.\\
                          \left. \left.
                          + |\gamma|^2  \left[ \exp(2\lambda_0 L) -1  \right]
                     \right\}                     
                    +2L \mathrm{Re} \left( \chi \bar{\gamma} \right)\right)  \nonumber \\
 + B \frac{\lambda_0}{2} \left\{  |\chi|^2  \left[1 - \exp(-2\lambda_0 L) \right] 
                            + |\gamma|^2  \left[ \exp(2\lambda_0 L) -1  \right]
                     \right\}                     
                    -2B\lambda_0^2 L \mathrm{Re} \left( \chi \bar{\gamma} \right)     \nonumber             
\end{multline}
if $k^2 \geq \frac{\xi_0}{\eta_0} \omega^2$ or
\begin{eqnarray}
&& e_k (\alpha) = (A+B|k|^2) \left( L(|\chi|^2  + |\gamma|^2) 
+ \frac{i}{\lambda_0} \mathrm{Im}\left\{ \chi  \bar{\gamma} \left[ 1 - \exp(-2\lambda_0 L) \right] \right\} \right)  \nonumber \\
&& + B L |\lambda_0|^2 \left( |\chi|^2  + |\gamma|^2 \right) + iB \lambda_0 \mathrm{Im}\left\{ \chi  \bar{\gamma} \left[ 1 - \exp(-2\lambda_0 L) \right] \right\}    \nonumber             
\end{eqnarray}
if $k^2 < \frac{\xi_0}{\eta_0} \omega^2$, in which
\begin{align*}
  &f(x)=(\lambda_0 \eta_0 - x) \exp(-\lambda_0 L) + (\lambda_0 \eta_0 + x) \exp(\lambda_0 L),\\
 &\chi(k,\alpha)=g_k \left(\frac{\lambda_0 \eta_0 - \lambda_1 \eta_1}{ f(\lambda_1 \eta_1) } -\frac{\lambda_0 \eta_0 - \alpha}{f(\alpha)}\right),\\
 &\gamma(k,\alpha)= g_k\left(\frac{\lambda_0 \eta_0 + \lambda_1 \eta_1}{f(\lambda_1 \eta_1)}-\frac{\lambda_0 \eta_0 + \alpha}{f(\alpha)} \right),
\end{align*}
where 
\begin{equation}\label{EqLambda0A}
 \left\lbrace \begin{array}{lcl}
\lambda_0 = \sqrt{k^2 - \frac{\xi_0}{\eta_0} \omega^2} & \mbox{if} & k^2 \geq \frac{\xi_0}{\eta_0} \omega^2, \\
\lambda_0 = i \sqrt{\frac{\xi_0}{\eta_0} \omega^2 - k^2} & \mbox{if} & k^2 \leq \frac{\xi_0}{\eta_0} \omega^2.
\end{array}\right. 
\end{equation}
\end{theorem}

\begin{proof}
First of all,
$$
e(\alpha):=A||u_0-u_2||^2_{L_2(\Omega_0)} + B ||\nabla(u_0-u_2)||^2_{L_2(\Omega_0)}
$$
can be decomposed as a sum of $e_k(\alpha)$
$$
e(\alpha):= \sum_{k=\frac{n\pi}{L}, n \in \mathbb{Z}} e_k (\alpha),
$$
with
$$
e_k(\alpha)= A||u_{0,k} - u_{2,k}||_{L_2(]-L,0[)}^2 + B ||\nabla (u_{0,k} - u_{2,k})||_{L_2(]-L,0[)}^2,
$$
where  we  have decomposed  decomposed $u_0$, $u_1$ and $u_2$ into modes in the $y$ direction,
denoting by $k$ the associated wave number. 

The mode $u_{0,k}$ solves
$$
\partial_{xx} u_{0,k} -\left( k^2 - \frac{\xi_0}{\eta_0} \omega^2 \right) u_{0,k} = 0,
$$
and thus
\begin{equation} \label{eq:expression_u1}
u_{0,k} (x) = A_0 \exp(\lambda_0 x) + B_0 \exp(-\lambda_0 x),
\end{equation}
where $\lambda_0$ is given in  Eq.~(\ref{EqLambda0A}).

The mode $u_{1,k}$ solves
$$
\partial_{xx} u_{1,k} -\left( k^2 - \frac{\tilde{\xi}_1}{\eta_1} \omega^2 \right) u_{1,k} = 0,
$$
and thus
\begin{equation} \label{eq:expression_generale_u2}
u_{1,k} (x) = A_1 \exp(\lambda_1 x) + B_1 \exp(-\lambda_1 x),
\end{equation}
where 
$$
\lambda_1^2 = k^2 -  \left( 1 + \frac{ai}{\xi_1\omega}\right) \frac{\xi_1}{\eta_1} \omega^2,
$$
so that
\begin{eqnarray}
\lambda_1 = 
  \frac{1}{\sqrt{2}} \sqrt{k^2 - \frac{\xi_1}{\eta_1} \omega^2
+ \sqrt{\left( k^2 - \frac{\xi_1}{\eta_1} \omega^2 \right)^2 + \left(\frac{a\omega}{\eta_1} \right)^2}} \nonumber \\
- \frac{i}{\sqrt{2}} \sqrt{\frac{\xi_1}{\eta_1} \omega^2 - k^2
+ \sqrt{\left( k^2 - \frac{\xi_1}{\eta_1} \omega^2 \right)^2 + \left(\frac{a\omega}{\eta_1} \right)^2}}. \nonumber
\end{eqnarray}
For large $L$, since $Re(\lambda_1)>0$, the value of $A_1$ tend to $0$, so that we may neglect the first contribution in the right-hand side
of  Eq.~(\ref{eq:expression_generale_u2}). Consequently we  consider the expression
\begin{equation} \label{eq:expression_finale_u2}
u_{1,k} (x) =  B_1 \exp(-\lambda_1 x).
\end{equation}
Continuity conditions  (\ref{eq:continuity_conditions}) and expressions  (\ref{eq:expression_u1}) and  (\ref{eq:expression_finale_u2}) imply the following relations
$$
A_0+B_0=B_1 \; \; , \; \;  \eta_0\lambda_0(A_0-B_0)= - \eta_1 \lambda_1 B_1,
$$
from which we infer that
$$
B_0=\frac{\lambda_0 \eta_0 + \lambda_1 \eta_1}{\lambda_0 \eta_0 - \lambda_1 \eta_1} A_0,
$$
and thus
$$
u_{0,k} (x) = A_0 \left[ \exp(\lambda_0 x) + \frac{\lambda_0 \eta_0 + \lambda_1 \eta_1}{\lambda_0 \eta_0 - \lambda_1 \eta_1} \exp(-\lambda_0 x) \right].
$$
The decomposition of the boundary condition  (\ref{eq:left_bc}) into Fourier modes implies
that $u_{0,k} (-L) = g_k$, which gives the final expression
\begin{equation} \label{eq:expression_u1_final}
u_{0,k} (x) = g_k \frac{\left[ (\lambda_0 \eta_0 - \lambda_1 \eta_1) \exp(\lambda_0 x) + (\lambda_0 \eta_0 + \lambda_1 \eta_1) \exp(-\lambda_0 x) \right]}{\left[ (\lambda_0 \eta_0 - \lambda_1 \eta_1) \exp(-\lambda_0 L) + (\lambda_0 \eta_0 + \lambda_1 \eta_1) \exp(\lambda_0 L) \right]} .
\end{equation}
Let us now turn to the expression of $u_{2,k}$. Since  equation  (\ref{eq:equation_u3}) is the same as that verified by $u_{0,k}$, both solutions have the same general form:
$$
u_{2,k} (x) = A_2 \exp(\lambda_0 x) + B_2 \exp(-\lambda_0 x).
$$
The Robin boundary condition  (\ref{eq:cl_robin_gamma}) on $\Gamma$ implies that
$$
\eta_0 \lambda_0 (A_2 - B_2) +\alpha (A_2 + B_2) = 0,
$$
which means that 
$$
u_{2,k} (x) = A_2 \left[ \exp(\lambda_0 x) + \frac{\lambda_0 \eta_0 + \alpha}{\lambda_0 \eta_0 - \alpha} \exp(-\lambda_0 x) \right].
$$
Application of the boundary condition  (\ref{eq:cl_dirichlet_u3_left})
implies the final expression
\begin{equation} \label{eq:expression_u3_final}
u_{2,k} (x) = g_k \frac{\left[ (\lambda_0 \eta_0 - \alpha) \exp(\lambda_0 x) + (\lambda_0 \eta_0 + \alpha) \exp(-\lambda_0 x) \right]}{\left[ (\lambda_0 \eta_0 - \alpha) \exp(-\lambda_0 L) + (\lambda_0 \eta_0 + \alpha) \exp(\lambda_0 L) \right]} .
\end{equation}
Using (\ref{eq:expression_u1_final}) and (\ref{eq:expression_u3_final}), we have that
\begin{equation} \label{diff_u1-u3}
(u_{0,k} - u_{2,k}) (x) = \chi(k,\alpha) \exp(\lambda_0 x) + \gamma(k,\alpha) \exp(-\lambda_0 x),
\end{equation}
where the coefficients $\chi$ and $\gamma$ are computed from (\ref{eq:expression_u1_final}) and (\ref{eq:expression_u3_final}).
In order to compute the $L_2$ norm of this expression, we must first compute the square of its modulus (by $\bar{\gamma}$ is denoted the complex conjugate of $\gamma$):
$$
| u_{0,k} - u_{2,k} |^2 (x) = |\chi|^2 |\exp(\lambda_0 x)|^2 + |\gamma|^2 |\exp(-\lambda_0 x)|^2
+2 \mathrm{Re} \left( \chi \bar{\gamma} \exp(\lambda_0 x) \overline{\exp(-\lambda_0 x)} \right).
$$
Note that, according to the values of $k$, the expression above may be simplified into
$$
| u_{0,k} - u_{2,k} |^2 (x) = |\chi|^2 \exp(2\lambda_0 x) + |\gamma|^2 \exp(-2\lambda_0 x)
+2 \mathrm{Re} \left( \chi \bar{\gamma} \right),
$$
if $k^2 \geq \frac{\xi_0}{\eta_0} \omega^2$, or
$$
| u_{0,k} - u_{2,k} |^2 (x) = |\chi|^2  + |\gamma|^2 
+2 \mathrm{Re} \left( \chi \bar{\eta} \exp(2\lambda_0 x) \right),
$$
if $k^2 < \frac{\xi_0}{\eta_0} \omega^2$.
Thus, we have for $k^2 \geq \frac{\xi_0}{\eta_0} \omega^2$
\begin{eqnarray*}
 \int_{-L}^0 | u_{0,k} - u_{2,k} |^2 (x) \dx =&
\frac{1}{2\lambda_0} \left\{  |\chi|^2  \left[1 - \exp(-2\lambda_0 L) \right] 
                            + |\gamma|^2  \left[ \exp(2\lambda_0 L) -1  \right]
                     \right\}\\                     
                    &+2L \mathrm{Re} \left( \chi \bar{\gamma} \right)\nonumber
\end{eqnarray*}
or, for $k^2 < \frac{\xi_0}{\eta_0} \omega^2$,
$$
\int_{-L}^0 | u_{0,k} - u_{2,k} |^2 (x) \dx = L(|\chi|^2  + |\gamma|^2) 
+ \frac{i}{\lambda_0} \mathrm{Im}\left\{ \chi \bar{\gamma} \left[ 1 - \exp(-2\lambda_0 L) \right] \right\}.
$$

Now, we also have to compute the $L_2$ norm of the gradient of $(u_{0,k} - u_{2,k})$.
Noting that 
$$
\nabla (u_{0,k} - u_{2,k}) = \left( \begin{array}{c}
\partial_x (u_{0,k} - u_{2,k}) \\
ik (u_{0,k} - u_{2,k})
\end{array}\right),
$$
it holds that
$$
\left|\nabla (u_{0,k} - u_{2,k}) \right|^2=|k|^2 | u_{0,k} - u_{2,k} |^2 + |\partial_x (u_{0,k} - u_{2,k})|^2.
$$
With  expression (\ref{diff_u1-u3}), it follows that
$$
|\partial_x (u_{0,k} - u_{2,k})|^2 = |\lambda_0|^2 \left[ |\chi|^2 \exp(2\lambda_0 x) 
+ |\gamma|^2 \exp(-2\lambda_0 x)
-2 \mathrm{Re} \left( \chi \bar{\gamma} \right) \right],
$$
if $k^2 \geq \frac{\xi_0}{\eta_0} \omega^2$, or
$$
|\partial_x (u_{0,k} - u_{2,k})|^2 = |\lambda_0|^2 \left[ |\chi|^2  + |\gamma|^2 
-2 \mathrm{Re} \left( \chi \bar{\gamma} \exp(2\lambda_0 x) \right) \right],
$$
if $k^2 < \frac{\xi_0}{\eta_0} \omega^2$, and thus
\begin{multline*}
 \int_{-L}^0 |\partial_x (u_{0,k} - u_{2,k})|^2 (x) \dx =
\frac{\lambda_0}{2} \left\{  |\chi|^2  \left[1 - \exp(-2\lambda_0 L) \right] 
                            + |\gamma|^2  \left[ \exp(2\lambda_0 L) -1  \right]
                     \right\}   \\                  
                    -2\lambda_0^2 L \mathrm{Re} \left( \chi \bar{\gamma} \right),
\end{multline*}
if $k^2 \geq \frac{\xi_0}{\eta_0} \omega^2$, or, if $k^2 < \frac{\xi_0}{\eta_0} \omega^2$,
$$
\int_{-L}^0 |\partial_x (u_{0,k} - u_{2,k})|^2 (x) \dx =
L |\lambda_0|^2 \left( |\chi|^2  + |\gamma|^2 \right) + i \lambda_0 \mathrm{Im}\left\{ \chi  \bar{\gamma} \left[ 1 - \exp(-2\lambda_0 L) \right] \right\}.
$$
Therefore, we can find $\alpha$ as the solution of the mentioned minimization problem.
\end{proof}

\begin{figure}[ht!]
\centering
\subfigure[$ \mathrm{Re} (\alpha)$]{
\includegraphics[width=0.3\linewidth]{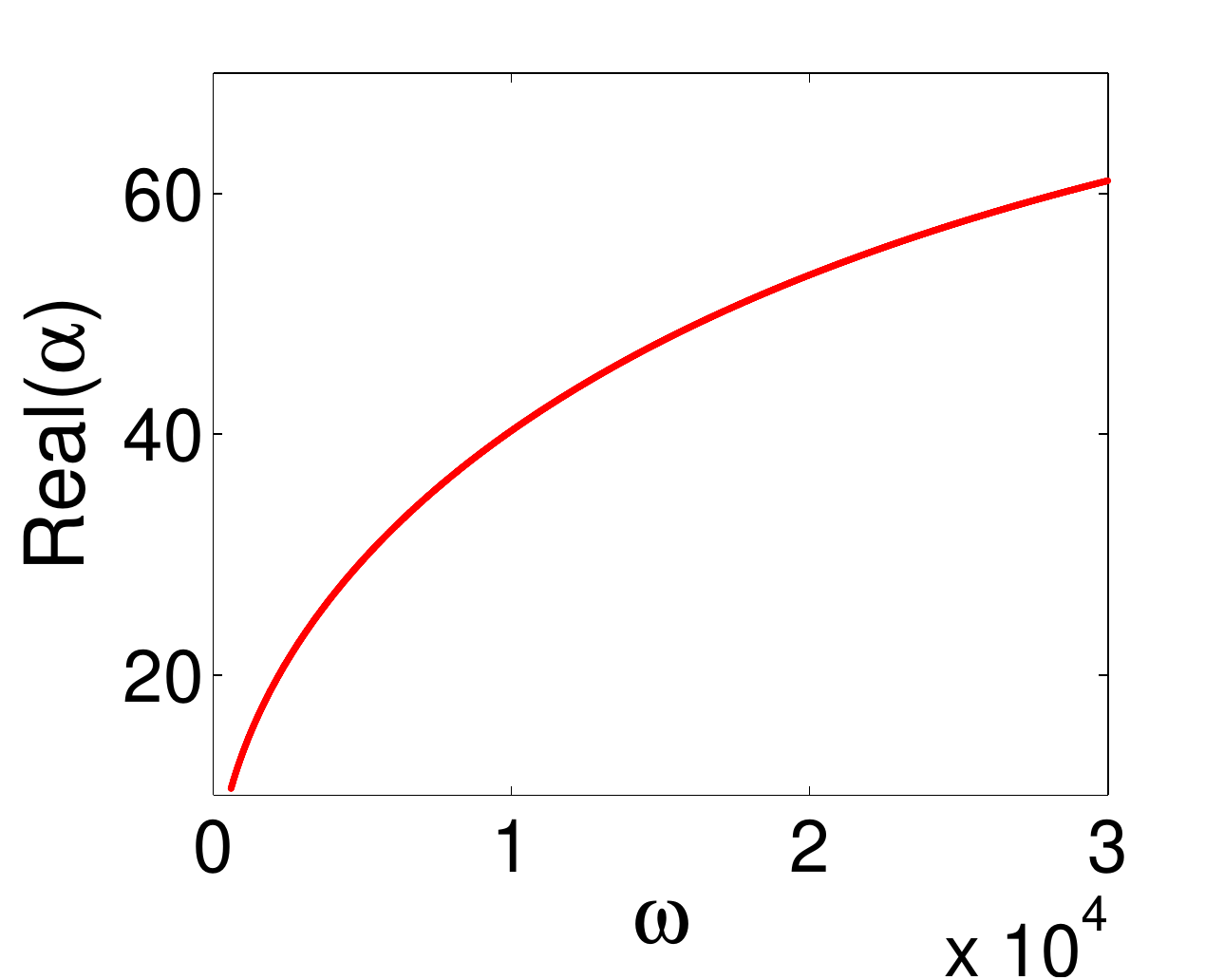}
}
\subfigure[$ \mathrm{Im} (\alpha)$]{
\includegraphics[width=0.3\linewidth]{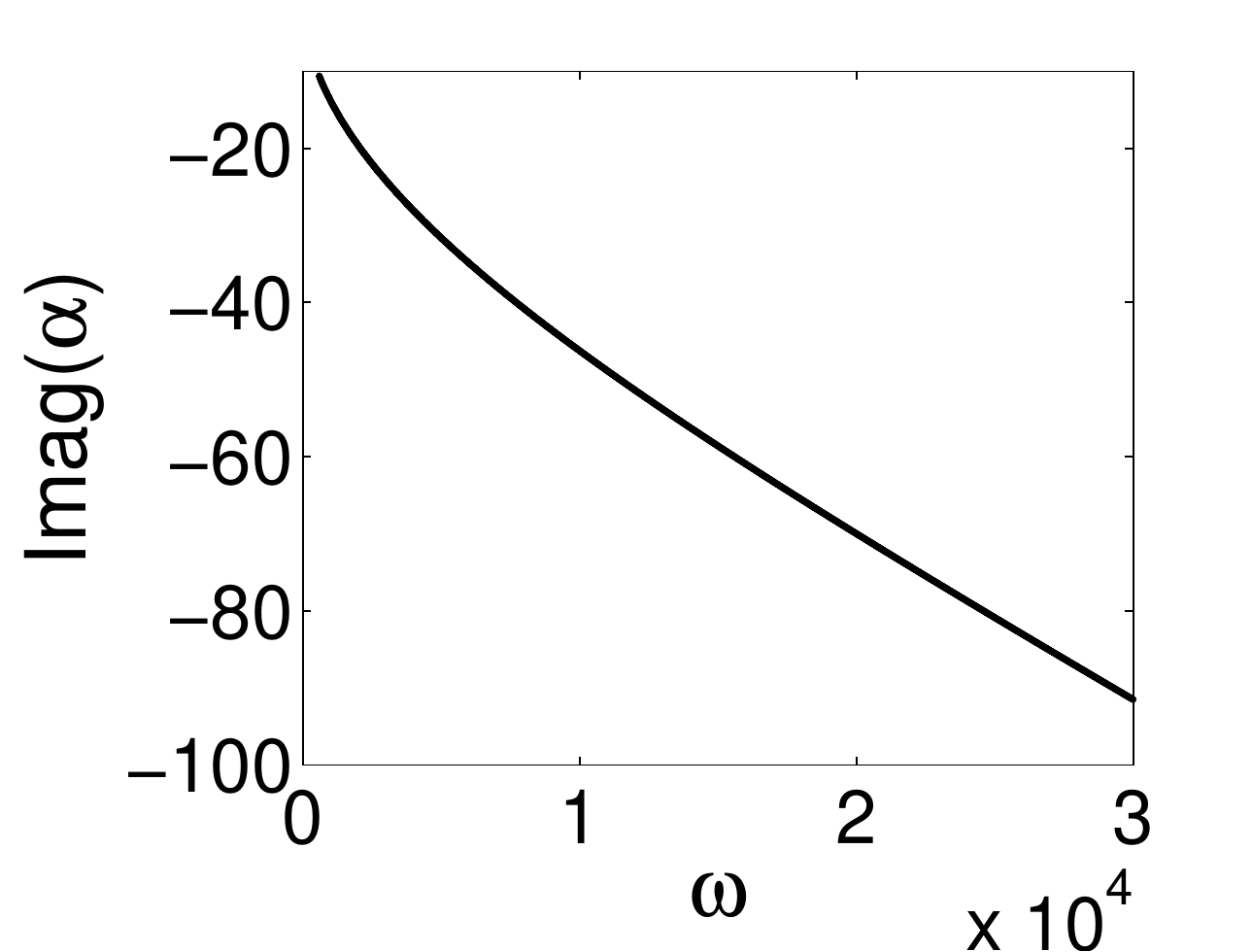}
}
\subfigure[error]{
\includegraphics[width=0.3\linewidth]{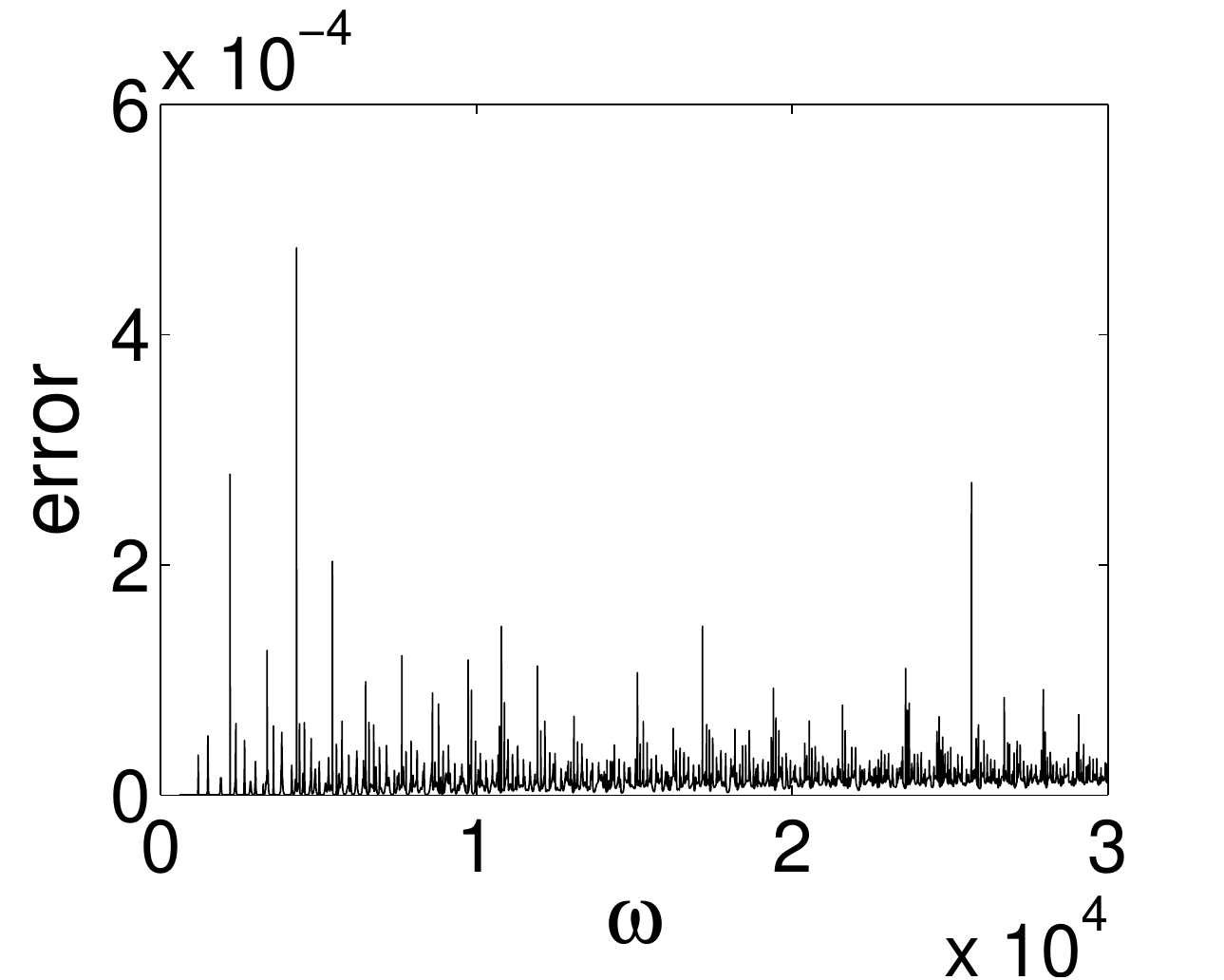}
}
\vspace*{8pt}

\caption{\label{FigAlpha}The real (top left) and imaginary (top right) parts of $\alpha$ and the sum of the errors $e_{\Delta x}$ (in the bottom) as function of frequencies $\omega\in[600,30000]$ calculated for the ISOREL porous material. }
\end{figure}

Since the minimization will be done numerically and since the sequence $(z,-z,z-z,\cdots)=z(\exp(i(j\Delta x)/\Delta x))$ is the highest frequency mode that can be reached on a grid of size $\Delta x$, then, in practice, the sum may be truncated to 
$$
e_{\Delta x}(\alpha):= \sum_{k=\frac{n\pi}{L}, n \in \mathbb{Z}, -\frac{L}{\Delta x} \leq n \leq \frac{L}{\Delta x} } e_k (\alpha).
$$
For the equations  (\ref{eq:Ap1})--(\ref{eq:Ap2}), we use the same coefficients as for problem \eqref{amortih} and take the values corresponding to a porous medium, called ISOREL, using in building insulation.
More precisely we assume: $\phi = 0.7$, $\gamma_p = 1.4$, $\sigma = 142300 N.m^{-4}.s$, $\rho_0 = 1.2 kg/m^3$, $\alpha_h = 1.15$, $c_0 = 340 m.s^{-1}$. %
We could find the value of $\alpha$ presented in Fig.~\ref{FigAlpha}.
\begin{remark}
Fig.~\ref{FigAlpha} allows us to compare the difference between two considered time-dependent models for the damping in the volume and for the damping on the boundary.
We see that $\operatorname{Re}(\alpha)$ is not a constant in general, but for $\omega\to +\infty$ $\operatorname{Im}(\alpha)$ is a linear function of $\omega$.
In this sense, the damping properties of two models are almost the same, but the reflection is more accurately considered by the damping wave equation in the volume. 
\end{remark}

\bibliographystyle{siamplain}
\label{bib:sec}
\bibliography{/home/anna/Documents/Statii/bibtex/biblio.bib}

\end{document}